\documentclass[11pt]{amsart}
\usepackage{graphicx}
\usepackage{amssymb,amscd}
\newtheorem{thm}{Theorem}[section]

\newtheorem{lemma}[thm]{Lemma}

\theoremstyle{definition}
\newtheorem{defn}[thm]{Definition}

\newtheorem{example}[thm]{Example}
\newtheorem{obs}[thm]{Observation}

\newcommand{\bi}{\begin{itemize}}
\newcommand{\ei}{\end{itemize}}
\newcommand{\be}{\begin{enumerate}}
\newcommand{\ee}{\end{enumerate}}
\newcommand{\bc}{\begin{center}}
\newcommand{\ec}{\end{center}}
\newcommand{\bt}{\begin{tabular}}
\newcommand{\et}{\end{tabular}}
\newcommand{\ba}{\begin{array}}
\newcommand{\ea}{\end{array}}

\newcommand{\Z}{\mathbb Z}

\newcommand{\PP}{{\mathcal P}}

\begin{document}

\title{Computing word length in alternate presentations of Thompson's group $F$}

\author[Matthew Horak]{Matthew Horak}
\address{Department of Mathematics, Statistics and Computer Science, University of Wisconsin-Stout,  Menomonie, WI 54751}
\email{horakm@uwstout.edu}

\author[Melanie Stein] {Melanie Stein}
\address{Department of Mathematics, Trinity College, Hartford, CT 06106}
\email{melanie.stein@trincoll.edu}

\author[Jennifer Taback] {Jennifer Taback}
\address{Department of Mathematics, Bowdoin College, Brunswick, ME 04011}
\email{jtaback@bowdoin.edu}

\thanks{The third author acknowledges partial support from
National Science Foundation grant DMS-0437481.}

\date{\today}

\begin{abstract}
We introduce a new method for computing the word length of an
element of Thompson's group $F$ with respect to a ``consecutive"
generating set of the form $X_n = \{x_0,x_1, \cdots ,x_n\}$, which
is a subset of the standard infinite generating set for $F$. We use
this method to show that $(F,X_n)$ is not almost convex, and has
pockets of increasing, though bounded, depth dependent on $n$.
\end{abstract}

\maketitle

\section{Introduction}

Many questions in geometric group theory investigate whether a
particular group has a given property.  An ideal answer involves a
determination of whether the group has the property with respect to
all, none or some generating sets.  There are few definitive answers
of this form when the group in question is Thompson's group $F$
because there is a single finite generating set with respect to
which the word length of group elements can be computed.  This
allows results for $F$ such as
\begin{thm}\label{thm:x0x1}
With respect to the generating set $\{x_0,x_1\}$, Thompson's group
$F$

\begin{enumerate} \item is not almost convex. \cite{CT}
\item has only pockets of depth two. \cite{CT2}
\end{enumerate}
\end{thm}
Ideally we would like to determine whether the group has these, or
other, properties with respect to any or no generating set, or list
exactly those generating set which yield these properties.

In this paper, we present a method for computing the word length of
elements of $F$ with respect to {\em consecutive} generating sets of
the form $X_n = \{x_0,x_1, \cdots ,x_n\}$.  This greatly expands the
set of generating sets in which we can compute the word length.  We
note that the three known methods of computing word length for the
generating set $\{x_0,x_1\}$, due to  Fordham \cite{F}, Guba
\cite{G}, and Belk-Brown \cite{BB}, are all special cases of our
procedure when $n=1$.

We use this method to prove the following theorems which extend the
results listed above in Theorem \ref{thm:x0x1}.

{\bf Theorem 6.1.} Thompson's group $F$ is not almost convex with
respect to the generating set $X_n = \{x_0,x_1, \cdots ,x_n\}$ .

{\bf Theorem 7.1.} For any $k \geq 1$, Thompson's group $F$ has
pockets of depth at least $k$ with respect to the generating set
$X_n = \{x_0,x_1, \cdots ,x_n\}$, for $n \geq 2k+2$.

In addition, we are able to provide an upper bound on the depth of
these pockets which is also dependent on $n$.

This paper is organized as follows.  The second section provides a
short introduction to Thompson's group $F$.  The third section
outlines and proves our procedure for computing word length,
although the proofs of the two main lemmas are deferred to sections
four and five.  Section six is devoted to the proof of Theorem
\ref{thm:notAC}, and in section seven we prove Theorem
\ref{thm:pockets} as well as an upper bound on pocket depth.

\section{Background on Thompson's group $F$}

We present a brief introduction to Thompson's group $F$ and refer
the reader to \cite{cfp} for a more detailed discussion.  This group
can be studied either as a finitely or infinitely presented group,
using the two standard presentations: $$\langle x_k, \ k \geq 0 |
x_i^{-1}x_jx_i = x_{j+1} \ if \ i<j \rangle$$ or, as it is
clear that $x_0$ and $x_1$ are sufficient to generate the entire
group,
 since powers of $x_0$ conjugate $x_1$ to $x_i$ for $i \geq 2$,
$$\langle x_0,x_1 |
[x_0x_1^{-1},x_0^{-1}x_1x_0],[x_0x_1^{-1},x_0^{-2}x_1x_0^2]
\rangle.$$  The relators in the infinite presentation are all a
consequence of the basic set of two relators given in the finite
presentation.

With respect to the infinite presentation, each element $g \in F$
can be written in normal form as $$g =x_{i_1}^{r_1}
x_{i_2}^{r_2}\ldots x_{i_k}^{r_k} x_{j_l}^{-s_l} \ldots
x_{j_2}^{-s_2} x_{j_1}^{-s_1} $$ with $r_i, s_i >0$, $i_1<i_2 \ldots
< i_k$ and $j_1<j_2 \ldots < j_l$. This normal form is unique if we
further require that when both $x_i$ and $x_i^{-1}$ occur, so does
$x_{i+1}$ or $x_{i+1}^{-1}$, as discussed by Brown and Geoghegan
\cite{BG}. We will use the term \emph{normal form} to mean this
unique normal form.

Elements of $F$ can be viewed combinatorially as pairs of finite
binary rooted trees, each with the same number $n$ of carets, called
tree pair diagrams. We define a {\em caret} to be a vertex of the
tree together with two downward oriented edges, which we refer to as
the left and right edges of the caret. The {\em right (respectively
left) child} of a caret $c$ is defined to be a caret which is
attached to the right (resp. left) edge of $c$.  If a caret $c$ does
not have a right (resp. left) child, we call the right (resp. left)
leaf of $c$ {\em exposed}.  Define the {\em level} of a caret
inductively as follows. The root caret is defined to be at level 1,
and the child of a level $k$ caret has level $k+1$, for $k \geq 1$.
The {\em left (resp. right) side} of a tree is defined to be the
maximal path of left (resp. right) edges beginning at the root
caret.

We number the leaves of each tree from left to right from $0$
through $n$, and number the carets in infix order from $1$ through
$n$.  The infix ordering is carried out by numbering the left child
of a caret $c$ before numbering $c$, and the right child of $c$
afterwards.

An element $g \in F$ is represented by an equivalence class of tree
pair diagrams, among which there is a unique reduced tree pair
diagram.  We say that a pair of trees is {\em unreduced} if when the
leaves are numbered from $0$ through $n$, there is a caret in both
trees with two exposed leaves bearing the same leaf numbers.  We
remove such pairs until no more exist, producing the unique {\em
reduced} tree pair diagram representing $g$.  See Figure
\ref{fig:reduced} for an example of  reduced and unreduced tree pair
diagrams representing the same group element.  The reduced tree pair
diagrams for $x_0$ and $x_n$ are given in Figure
\ref{fig:generators}.  When we write $g = (T,S)$, we are assuming
that this is the unique reduced tree pair diagram representing $g
\in G$.

\begin{figure}
\begin{center}
\includegraphics[width=5in]{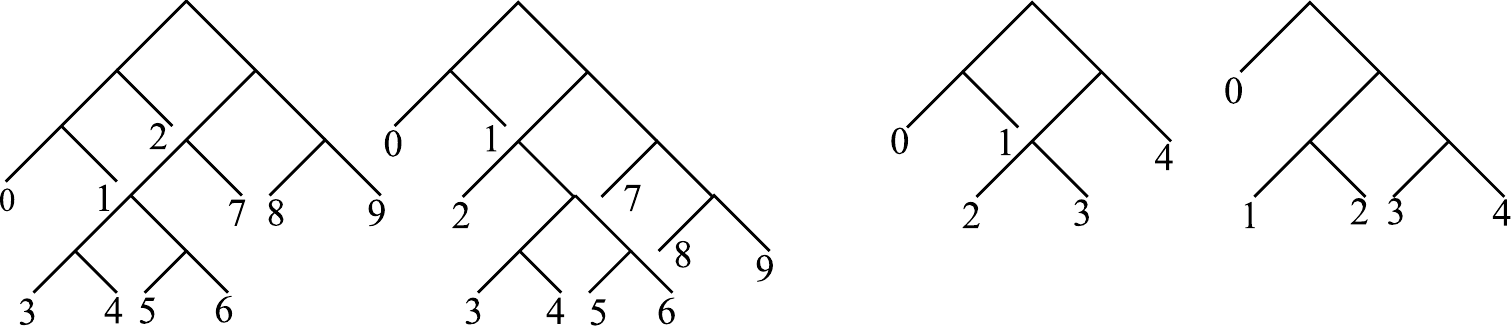}
\caption{An example of an unreduced and then a reduced tree pair
diagram representing the same group element.} \label{fig:reduced}
\end{center}
\end{figure}

\begin{figure}
\begin{center}
\includegraphics[width=5in]{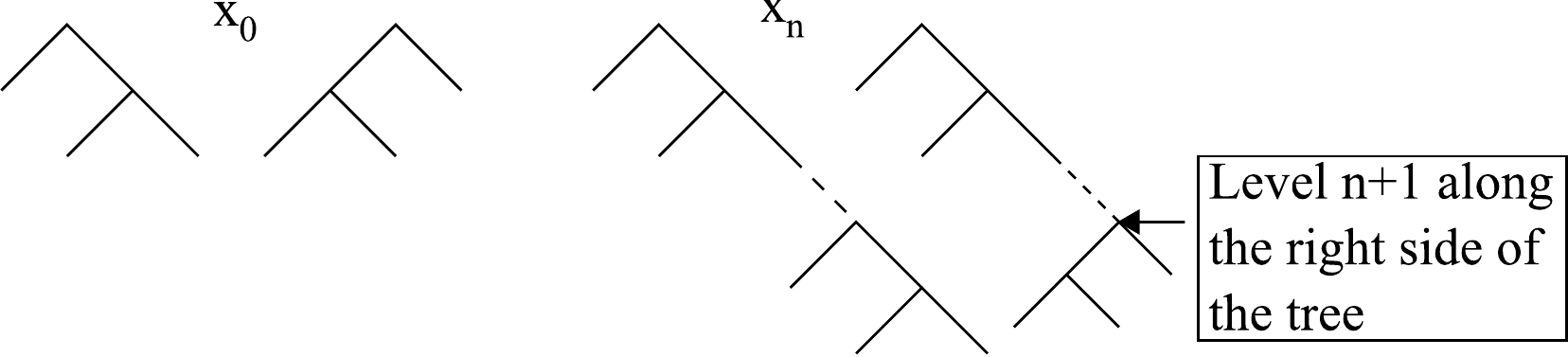}
\caption{The reduced tree pair diagrams representing the generators
$x_0$ and $x_n$ of $F$.} \label{fig:generators}
\end{center}
\end{figure}

The equivalence of these two interpretations of Thompson's group is
given using the normal form for elements with respect to the
standard infinite presentation, and the concept of leaf exponent. In
a single tree $T$ whose leaves are numbered from left to right
beginning with $0$, the {\em leaf exponent} $E(k)$ of leaf number
$k$ is defined to be the integral length of the longest path of left
edges from leaf $k$ which does not reach the right side of the tree.
Figure \ref{fig:leafexp} gives an example of a tree whose leaf
exponents are computed.

Given a reduced tree pair diagram $(T,S)$ representing $g \in F$,
compute the leaf exponents $E(k)$ for all leaves $k$ in $T$,
numbered $0$ through $n$.   The negative part of the normal form for
$g$ is then $x_n^{-E(n)} x_{n-1}^{-E(n-1)} \cdots x_1^{-E(1)}
x_0^{-E(0)}$. We compute the exponents $E(k)$ for the leaves of the
tree $S$ and thus obtain the positive part of the normal form as
$x_0^{E(0)} x_{1}^{E(1)} \cdots x_m^{E(m)}$.
 Many of these exponents will be $0$, and after deleting these, we can index the remaining
terms to correspond to the normal form given above, following
\cite{cfp}.  As a result of this process, we often denote a tree
pair diagram as $(T_-,T_+)$, since the first tree in the pair
determines the terms in the normal form with negative exponents, and
the second tree determines those terms with positive exponents.  We
refer to $T_-$ as the negative tree in the pair, and $T_+$ as the
positive tree.

\begin{figure}
\begin{center}
\includegraphics[scale=.95]{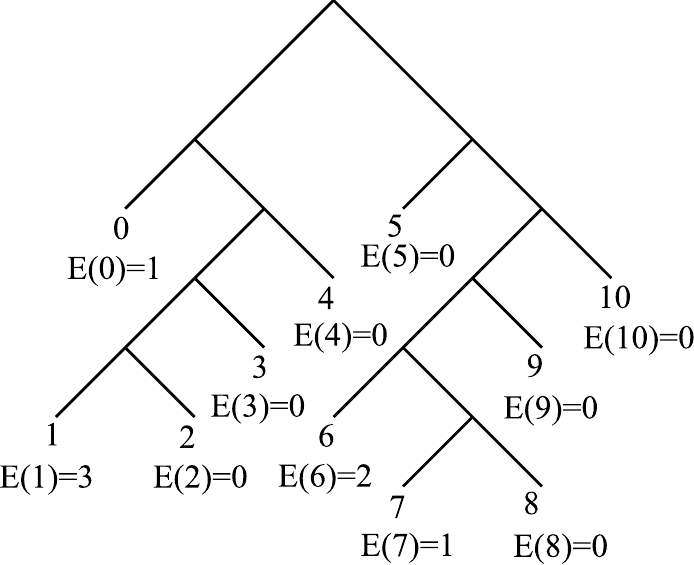}
\caption{An example of a tree with leaf exponents computed.}
\label{fig:leafexp}
\end{center}
\end{figure}

Group multiplication is defined as follows when multiplying two
elements represented by tree pair diagrams.  Let $g = (T_-,T_+)$ and
$h = (S_-,S_+)$.  To form the product $gh$, we take unreduced
representatives of both elements, $(T'_-,T'_+)$ and $(S'_-,S'_+)$,
respectively, in which $S'_+ = T'_-$.  The product is then
represented by the (possibly unreduced) pair of trees $(S'_-,T'_+)$.
An example of the unreduced representatives necessary to perform
group multiplication is given in Figure \ref{fig:x1mult}, where the
trees can be used to form the product $gx_1$, for $g =
x_0x_1x_4^2x_5^{-1}x_3^{-1}x_2^{-2}x_0^{-1}$.

\begin{figure}
\begin{center}
\includegraphics[width=5.5in]{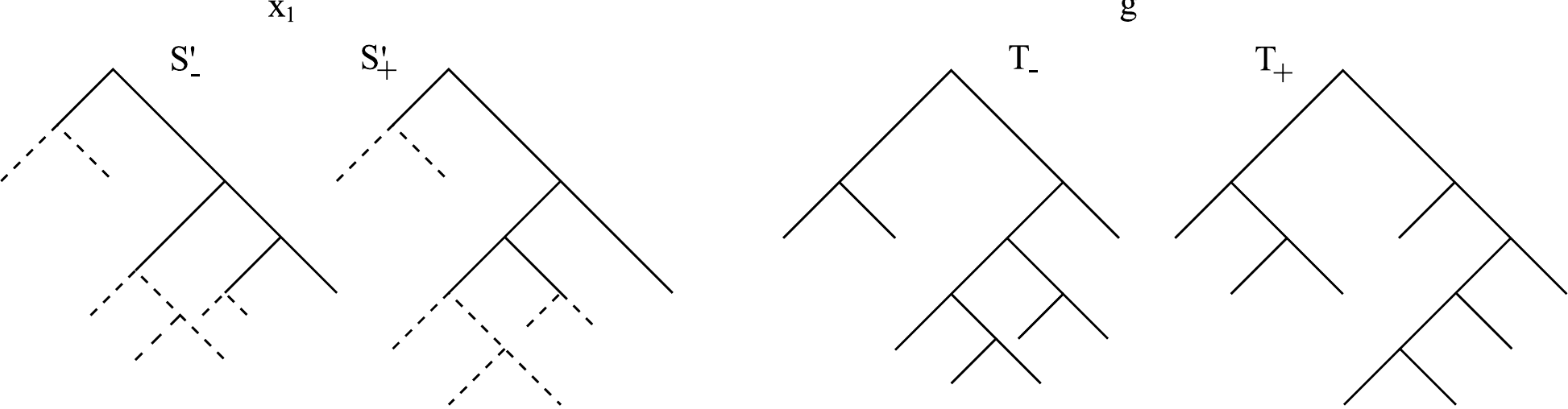}
\caption{To multiply $g
=x_0x_1x_4^2x_5^{-1}x_3^{-1}x_2^{-2}x_0^{-1}$ by the generator
$x_1=(S_-,S_+)$, we use an unreduced representative of $x_1$,
pictured above.  Dashed carets indicate the carets added in order to
perform the multiplication.} \label{fig:x1mult}
\end{center}
\end{figure}

\section{Computing word length with respect to a consecutive
generating set} In this section, we describe a method for computing
the word length of elements of $F$ with respect to a {\em
consecutive} generating set of the form $X_n = \{x_0,x_1, \cdots
,x_n\}$, which is a subset of the standard infinite presentation for
$F$. In the case $n=1$, there are three known formulae for computing
word length, due to Fordham \cite{F}, Guba \cite{G}, and Belk and
Brown \cite{BB}. We end this section with a comparison of these
methods, and translate the terminology of each into that of the
present paper.

Below we present our method for computing word length, along with a
detailed example.  The proof that this method actually computes the
word length of group elements follows the outline of Fordham's
proof, and we apply a lemma from \cite{F} as the main step in our
proof.  We then require two technical lemmas to show that the
conditions in Fordham's lemma are fulfilled, and we defer the proofs
of these lemmas to Sections \ref{sec:lemma1} and \ref{sec:lemma2}
below.

Let $T$ be a finite rooted binary tree with $n$ carets, in which we
number the carets from $1$ through $n$ in infix order. We use the
infix numbers as names for the carets, and the statement $p<q$ for
two carets $p$ and $q$ simply expresses the relationship between the
infix numbers. A caret is said to be a right (resp. left) caret if
one of its edges lies on the right (resp. left) side of $T$.  The
root caret can be considered either left or right.  A caret which is
neither left nor right is called an interior caret.

Our formula for the word length of elements $g \in F$ with respect
to the generating set $X_n = \{x_0,x_1, \cdots ,x_n\}$ has two
components. The first we call $l_{\infty}(g)$, as it is the word
length of $g$ with respect to the infinite generating set $\{x_i | i
\geq 0\}$ for $F$.  This quantity is simply the number of carets in
the reduced tree pair diagram representing $g$ which are not right
carets.  The difference between $l_{\infty}(g)$ and the word length
$l_n(g)$ is measured by what we refer to as the {\em penalty}
weight, denoted $p_n(g)$.  Twice this penalty weight is the second
component of our word length formula.

The intuition for this formula comes from the effect that
multiplication by a generator has on a tree pair diagram
$(T_-,T_+)$.  One can view multiplication by each generator as
performing a proscribed combinatorial rearrangement of the subtrees
of $T_-$ or $T_+$.  The rearrangement of these subtrees induced by
multiplication by $x_0$ and $x_2$ is shown explicitly in Figure
\ref{fig:mult}, and is analogous for multiplication by $x_n$ with
$n=1$ or $n >2$.

\begin{figure}
\begin{center}
\includegraphics[width=5in]{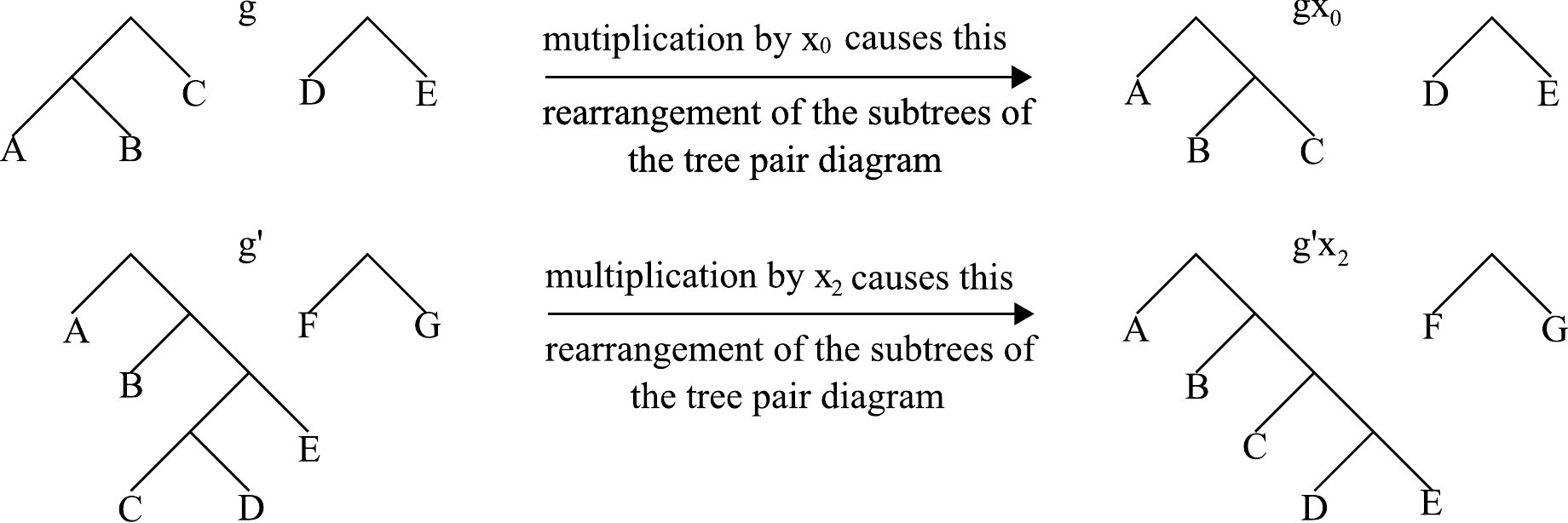}
\caption{The combinatorial rearrangement of the subtrees of the tree
pair diagrams representing elements $g$ and $g'$ of $F$ induced by
multiplication by $x_0$ and $x_2$ respectively. The letters $A$
through $G$ represent possibly empty subtrees of the tree pair
diagram.} \label{fig:mult}
\end{center}
\end{figure}

In creating a minimal length representative for $g \in F$, whose
length is necessarily the word length of $g$, there are some
arrangements of carets in $T_-$ or $T_+$ which may be harder to
produce using the combinatorial rearrangements available with the
given generators.  This incurs a ``penalty" contribution to the
length of the word.  Determining this penalty contribution $p_n(g)$
to the word length lies at the heart of our method.

We begin by distinguishing a particular type of caret in a single
tree.  Caret types are central to the length formulae of Fordham
\cite{F} and Belk-Bux \cite{BB}.  While they require, respectively,
seven and four caret types, we define a single one which is
sufficient for our proofs below.

\begin{defn}
Caret $p$ in a tree $T$ has {\em type N} if caret $p+1$ is an
interior caret which lies in the right subtree of $p$.
\end{defn}

We use this definition to describe certain carets in the tree pair
diagram for $g \in F$ which we call {\em penalty carets} as they
help determine the penalty contribution to the word length $l_n(g)$.
Let $g \in F$ have a reduced tree pair diagram $(T_-,T_+)$ in which
the carets are numbered in infix order.

\begin{defn} Caret $p$ in a tree pair diagram $(T_-,T_+)$ is a {\it penalty caret} if either
\begin{enumerate}
\item $p$ has type $N$ in either $T_-$ or $T_+$, or
\item $p$ is a right caret in both $T_-$ and $T_+$ and caret $p$ is not the final caret in the tree pair
diagram.
\end{enumerate}
\end{defn}

To compute the penalty contribution to the word length for a given
$g=(T_-,T_+) \in F$ we use the following procedure, which will be
made precise in Section \ref{sec:penaltytree}.  Using a notion of
caret adjacency defined below, we take the two trees $T_-$ and $T_+$
and construct a single tree $\PP$, called a {\em penalty tree},
whose vertices correspond to a subset of the carets of $T_-$ and
$T_+$, necessarily including the penalty carets.  This tree is
assigned a {\em weight} according to the arrangement of its
vertices. Minimizing this weight over all possible penalty trees
that can be constructed using the adjacencies between the carets of
$T_-$ and $T_+$ yields the penalty contribution $p_n(g)$ to the word
length $l_n(g)$.  We will prove the following theorem.

\begin{thm}\label{thm:length}
For every $g\in F$, the word length of $g$ with respect to the
generating set $X_n = \{x_0,x_1, \cdots ,x_n \}$ is given by the
formula
$$l_n(g) = l_\infty(g)+2p_n(g)$$
where $l_{\infty}(g)$ is the number of carets in the reduced tree
pair diagram for $g$ which are not right carets, and $p_n(g)$ is the
penalty weight.
\end{thm}

\subsection{Constructing a penalty tree}  \label{sec:penaltytree}Constructing penalty trees
for $g \in F$ requires a concept of directed caret adjacency, which
is an extension of the infix order.  To define the concept of
adjacency between carets in a tree $T$, we view each caret as a
space rather than an inverted v.  The point of intersection of the
left and right edges of the caret naturally splits the boundary of
this space into a left and right component. The space is bounded on
the right (resp. left) by a {\em generalized right (resp. left)
edge}. The generalized right (resp. left) edge may consist of actual
left (resp. right) edges of other carets in the tree, in addition to
the actual right (resp. left) edge of the caret itself. For example,
in Figure \ref{fig:adjacency}, the spaces which we consider as
carets are shaded, and the generalized left edge of caret 9 includes
the right edges of carets 7 and 8.

\begin{figure}
\begin{center}
\includegraphics[scale=.75]{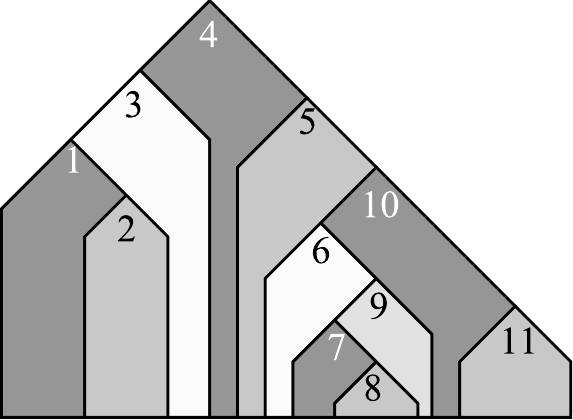}
\caption{The spaces corresponding to the different carets are
shaded.  These spaces are used to define the notion of caret
adjacency.} \label{fig:adjacency}
\end{center}
\end{figure}

Let $p$ and $q$ denote carets in a tree pair $(T_-,T_+)$, that is,
$p$ corresponds to a pair of carets, one in $T_-$ and one in $T_+$,
each with infix number $p$, and the same is true for $q$.
Additionally, assume  $p < q$. We say that $p$ is adjacent to $q$,
written $p\prec q$, if there is a caret edge, in either $T_-$ or
$T_+$, which is both part of the generalized right edge of caret $p$
and the generalized left edge of caret $q$. We equivalently say that
{\em traversing} the generalized left edge of caret $q$ takes you to
caret $p$ in at least one tree.  It is always true that carets $p$
and $p+1$ satisfy $p \prec p+1$.  Although the ordering of carets
given by infix number is not symmetric but is transitive, the notion
of caret adjacency is neither symmetric nor transitive.

We introduce a dummy caret denoted $v_0$ which is adjacent to all
left carets in both $T_-$ and $T_+$.  One can think of $v_0$ as
being the space to the left of the left side of each tree.  We now
construct a penalty tree $\PP$ corresponding to the pair of trees
$(T_-,T_+)$, which has this dummy caret $v_0$ as its root, according
to the following rules.

\begin{enumerate}
\item The vertices of $\PP$ are a subset of the carets in the tree pair
diagram, which we refer to by infix numbers: $0=v_0,1,2, \cdots ,k$,
always including $v_0$.
\item If a directed edge is drawn from vertex $p$ to vertex $q$ in
$\PP$ then we must have $p \prec q$.
\item There is a vertex for every penalty caret in $(T_-,T_+)$.
\item Each leaf of $\PP$ corresponds to a penalty caret of
$(T_-,T_+)$.  The only exception to this is when $\PP$ consists only
of the root $v_0$ and no edges.
\end{enumerate}
The penalty tree $\PP$ is oriented in the sense that there is a
unique path from $v_0$ to every vertex $p \in \PP$, and if this path
passes through vertices $v_0, p_1, p_2, \ldots, p_{i}=p$ then we
must have $v_0 \prec p_1 \prec \cdots \prec p_i=p$.  Two vertices
$p, q$ in the tree are comparable if there is either a path $p=w_1,
w_2, \ldots, w_{i+1}=q$ or $q=w_1, w_2, \ldots, w_{i+1}=p$ with $w_j
\prec w_{j+1}, \forall j=1, \ldots, i+1$, and in this case we say
$d_{\PP}(p,q)=i$.

\begin{figure}
\begin{center}
\includegraphics[width=5in]{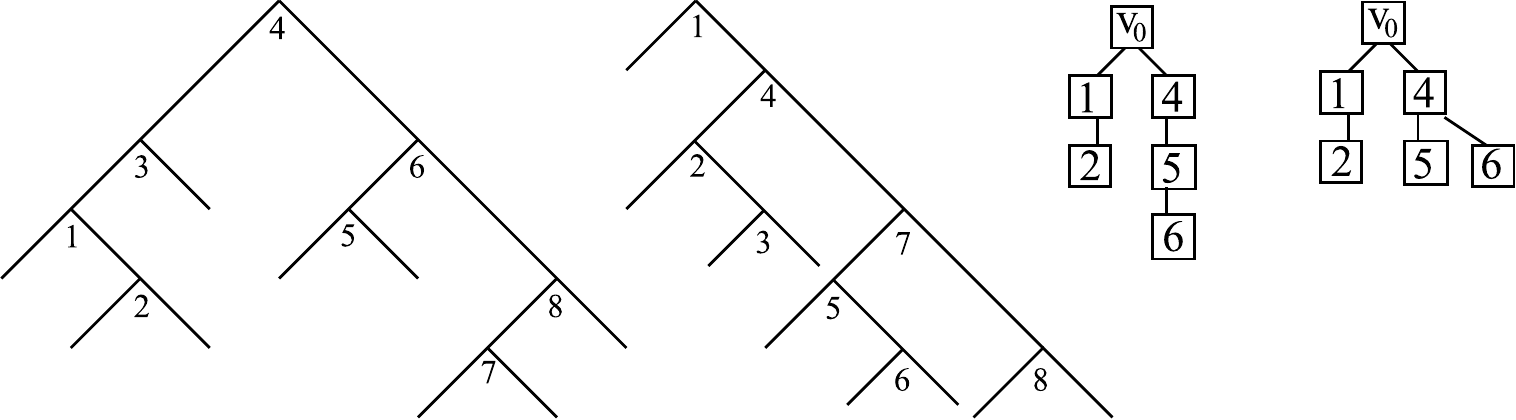}
\caption{An example of two penalty trees associated to the same
group element, whose carets are numbered in infix order from $1$
through $8$.} \label{fig:pen-trees}
\end{center}
\end{figure}

When working with these penalty trees, we often abuse notation and
refer to the edge between $p$ and $q$ as $p \prec q$, and
conversely, will sometimes refer to an adjacency $p \prec q$ which
exists in a tree pair diagram as an edge, meaning it can give rise
to an edge in a penalty tree.  Also, we call an edge $p \prec q$
both ``an edge out of $p$" and ``an edge into $q$."

The penalty weight of a penalty tree is bounded above by the number
of vertices on the tree, but not all vertices on the tree contribute
to the weight. More precisely, we define:
\begin{defn}The {\em n-penalty weight} $p_n(\PP)$ of a penalty tree $\PP$
associated to $g = (T_-,T_+) \in F$ is the number of vertices $v_i
\in \PP$ such that $d_{\PP}(0, v_i) \geq 2$ and there exists a leaf
$l_i$ in $\PP$ with $d_{\PP}(v_i, l_i) \geq n-1$. These vertices are
called the {\em weighted} carets.
\end{defn}

To compute the penalty contribution $p_n(g)$ to the word length
$l_n(g)$ for $g \in F$, we must minimize the penalty weight over all
penalty trees associated to $g$.
\begin{defn} For an element $g \in F$, define the penalty contribution $p_n(g)$ to the word length $l_n(g)$
by $$p_n(g)= min \{p_n(\PP)| \PP \mbox{ is a penalty tree for }
g=(T_-,T_+) \}$$
\end{defn}

This definition brings us to the statement of Theorem
\ref{thm:length}, which presents the formula $l_n(g) = l_{\infty}(g)
+ 2 p_n(g)$.  We call any penalty tree for $g$ which realizes
$p_n(g)$ a {\em minimal penalty tree}.

Computing the penalty contribution $p_n(g)$ for any $g=(T_-,T_+) \in
F$ can be quite difficult, as there may be a large number of
possible penalty trees based on the caret adjacencies present in
$T_-$ and $T_+$.  In Sections \ref{sec:notAC} and \ref{sec:pockets}
we present families of group elements where the penalty trees with
minimal penalty weight can be determined based on features of the
original tree pair diagrams.  One such feature which greatly
simplifies the computation of $p_n$ is recorded in the following
observation.

\begin{obs} \label{O:bottleneck}
Let $g \in F$ be represented by the reduced tree pair diagram $(T_-,
T_+)$. If $(T_-,T_+)$ contains two penalty carets $p \prec q$, where
in both trees, $p$ is a right caret which is not type $N$,
 then on any penalty tree $\PP$ for $g$, the unique
path from $v_0$ to $q$ must contain the caret $p$.
\end{obs}

\begin{obs}\label{obs:not-type-N}
Observation \ref{O:bottleneck} can be generalized as follows.  If
caret $p$ does not have type $N$ in either $T_-$ or $T_+$, then the
only caret $v$ with $p \prec v$ is $v=p+1$.
\end{obs}

The following lemma states that left carets in $T_-$ and $T_+$ can
never contribute to the penalty weight of a penalty tree.

\begin{lemma}\label{lemma:leftcarets}
Let $w = (T_-,T_+)$ be an element of $F$, and $p$ a caret which is a
left caret in either $T_-$ or $T_+$.  Then $p$ is not a weighted
caret in any minimal penalty tree for $w$.
\end{lemma}

\begin{proof}
Suppose that ${\mathcal P}$ is a minimal penalty tree for $w$ in
which $p$ is a vertex that carries weight.  We construct a new
minimal penalty tree ${\mathcal P}'$ for $w$ in which $p$ in not a
weighted caret.  In ${\mathcal P}$, let $c$ be the vertex which is
the parent of $p$.  Since $p$ is weighted in ${\mathcal P}$, we know
that $c$ is not the root caret of ${\mathcal P}$.

To construct $\PP'$, begin with ${\mathcal P}$ and remove the edge
$c \prec p$.  Attach vertex $p$, and its subtrees via the adjacency
$v_0 \prec p$, which arises from the fact that $p$ is a left caret
in either $T_-$ or $T_+$, and call the resulting tree $\PP'$.  Thus
we see that $p_n({\mathcal P}') < p_n({\mathcal P})$, since $p$ is
no longer a weighted caret in ${\mathcal P}'$.  This contradicts the
fact that ${\mathcal P}$ was a minimal penalty tree for $w$, and the
lemma follows.
\end{proof}

We now address the question of whether a minimal penalty tree
consists entirely of carets corresponding to left and penalty carets
in the tree pair diagram for $g \in F$.  We show directly that such
a penalty tree can always be constructed when $n=1$, and note that
this fact follows from a result of Guba \cite{G} discussed in
Section \ref{sec:other-methods} below.  When $n >1$, this need not
be the case, as we illustrate with an example below.

\begin{lemma}\label{lemma:min-tree}
In the case $n=1$, a minimal penalty tree $\PP$ can always be
constructed for $g=(T_-,T_+) \in F$ all of whose vertices correspond
to left carets or penalty carets in the tree pair diagram.
\end{lemma}

\begin{proof}
It follows from Lemma \ref{lemma:leftcarets} that left carets in
either tree can be assumed to be adjacent to $v_0$ in any minimal
penalty tree. Let ${\mathcal P}$ be a penalty tree for $w=(T_-,T_+)
\in F$ in which all carets in $\PP$ which are left in either $T_-$
or $T_+$ are adjacent to $v_0$ in $\PP$, and suppose that ${\mathcal
P}$ contains a vertex $v_i$ corresponding to a caret $v_i$ in
$(T_-,T_+)$ which is neither a penalty caret nor a left caret.  If
$v_i$ is a leaf in $\PP$, simply delete it. If not, Observation
\ref{obs:not-type-N} implies that the only caret $v$ with $v_i \prec
v$ is the caret immediately following caret $v_i$ in the infix
order. Call this caret $v_{i+1} = v_i+1$. Note that $v_{i+1}$ is not
a left caret, since it is not connected by an edge in $\PP$ to
$v_0$. Since $v_i$ is not a leaf of $\PP$, it follows that the one
and only edge out of $v_i$ on $\PP$ is $v_i \prec v_{i+1}$. Delete
both the edge $v_i \prec v_{i+1}$ and the vertex $v_i$ from
${\mathcal P}$, and attach the vertex $v_{i+1}$ and any subtree of
${\mathcal P}$ having it as a root, as follows.

Since $v_i$ in an interior caret with an exposed right leaf in at
least one of $T_-$ or $T_+$, without loss of generality we assume
this in $T_-$.  It follows that the adjacencies determined by the
actual (not generalized) left edges of $v_i$ and $v_{i+1}$ must
connect them to a single caret $c$ of type $N$. See Figure
\ref{fig:guba-pf} for two possible configurations of these carets.
We use the adjacency $c \prec v_{i+1}$ to reattach the subtree of
${\mathcal P}$ whose root is $v_{i+1}$ to the penalty tree. Thus we
have created a new penalty tree which does not contain the vertex
$v_i$. We can repeat this process until all non-penalty carets of
$(T_-,T_+)$ are removed from ${\mathcal P}$. Hence, $p_1(g)$ is
simply the number of penalty carets in which neither caret in the
pair is a left caret.
\end{proof}

\begin{figure}
\begin{center}
\includegraphics[width=2.5in]{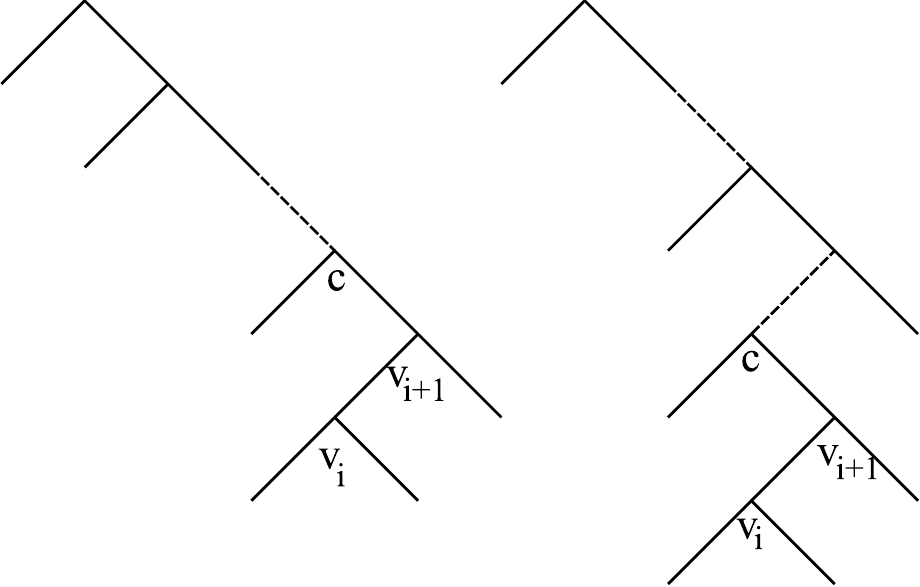}
\caption{Two possible configurations of the carets $c, v_i$ and
$v_{i+1}$ used to show that a penalty tree can always be constructed
using only vertices corresponding to left and penalty carets in
$(T_-,T_+)$ when $n=1$.} \label{fig:guba-pf}
\end{center}
\end{figure}

We conclude this section with two examples.  The first contrasts the
situations $n=1$ and $n > 1$, and the second illustrates the
computation of the word length $l_2(g)$ for a particular group
element $g \in F$.

\begin{example}
We first present an example contrasting the cases $n=1$ and $n > 1$.
We proved above that when $n=1$, a minimal penalty tree for $g \in
F$ can always be constructed using only penalty carets and left
carets. Although one can always construct a penalty tree for $g$
consisting only of penalty and left carets, for $n \geq 2$ this tree
may not be minimal. It may be the case that a penalty tree must
include some non-penalty carets in order to realize $p_n(g)$. The
following example illustrates this.

Consider $g=x_1x_2x_5x_6x_3^{-2}x_2^{-1}$ and the generating set
$X_3 = \{x_0,x_1,x_2,x_3\}$.  This element is depicted in Figure
\ref{fig:pen-trees2}, and we see that $l_{\infty}(g) = 7$.  Since
$g$ can be written as a word $x_3x_1x_2x_3^{-2}x_2^{-1}x_3$ of
length seven, we must have $p_3(g)=0$.  We see that the carets with
infix numbers $1,2,3,5$ and $6$ are penalty carets in the tree pair
diagram for $g$.  It is possible to make a penalty tree for $g$
using only these carets, but that tree will have penalty weight
equal to one.  In order to make a penalty tree with total weight
zero, we must add caret $4$ as a vertex.  These two penalty trees
are drawn in Figure \ref{fig:pen-trees2}.

\begin{figure}
\begin{center}
\includegraphics[width=5in]{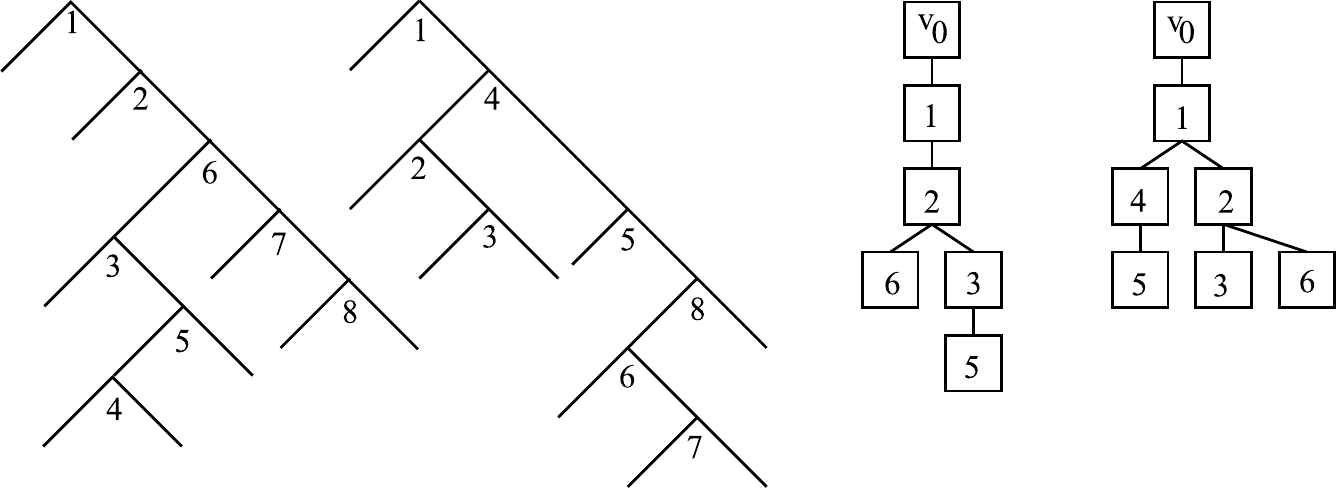}
\caption{The element $g = x_1x_2x_5x_6x_3^{-2}x_2^{-1}$, along with
two penalty trees: a non-minimal one which uses only penalty carets
as vertices, and a minimal one which requires the addition of a
vertex not corresponding to a penalty caret, both with respect to
the generating set $X_3$.} \label{fig:pen-trees2}
\end{center}
\end{figure}
\end{example}

\begin{example}
We now present an example in which we compute the word length of $$g
= x_0x_1^2 x_4 x_5^2 x_8 x_9^2 x_{12} x_{13}^2 x_{14}^{-1}
x_{12}^{-2} x_{10}^{-1} x_8^{-2} x_6^{-1} x_4^{-2} x_2^{-1}
x_0^{-2}$$ with respect to $X_2$. The tree pair diagram for this
element is given in Figure \ref{fig:example}. We see that
$l_{\infty}(g) = 24$, and begin the construction of a minimal
penalty tree $\PP$ by identifying the penalty carets to be those
numbered $1,2,4,5,6,8,9,10,12,13,14$.  We first note that any path
of adjacencies connecting penalty carets with infix numbers greater
than $8$ with $v_0$ must include the vertex $8$, as mentioned in
Observation \ref{O:bottleneck}. This ensures that carets $8$ and
$12$ will correspond to weighted penalty carets in any minimal
penalty tree for $g$, and leads to the construction of the minimal
penalty tree $\PP$ given in Figure \ref{fig:example}. We see that
$p_2(\PP) = p_2(g) = 2$, and compute $l_2(g) = 28$.

\begin{figure}
\begin{center}
\includegraphics[width=5in]{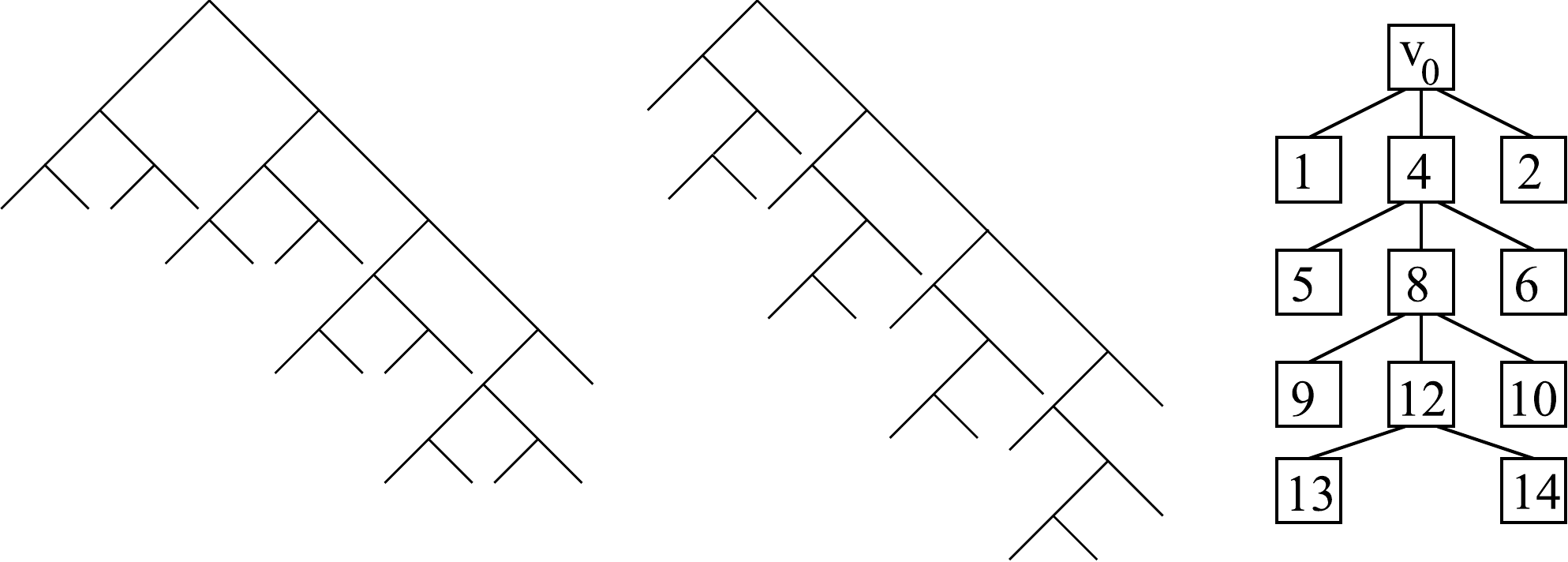}
\caption{The element $g = x_0x_1^2 x_4 x_5^2 x_8 x_9^2 x_{12}
x_{13}^2 x_{14}^{-1} x_{12}^{-2} x_{10}^{-1} x_8^{-2} x_6^{-1}
x_4^{-2} x_2^{-1} x_0^{-2}$, along with a minimal penalty tree for
$g$. With respect to $X_2$, we compute $p_2(g) = 2$, as only carets
$8$ and $12$ are weighted, and thus $l_2(g) = 28$.}
\label{fig:example}
\end{center}
\end{figure}
\end{example}

\subsection{Comparison with known methods when $n=1$}
\label{sec:other-methods}

In the case $n=1$, Fordham \cite{F}, Guba \cite{G}, and Belk and
Brown \cite{BB} have all provided formulas for $l_1(g)$. Our
formula, restricted to the case $n=1$, is seen below to be a
streamlined version of these methods.

Guba \cite{G} considers $F$ as a diagram group, and elements of $F$
are then infinite diagrams.  The cells of a diagram correspond
precisely to the carets in a tree pair diagram which are not right
carets. Furthermore, his {\em special vertices} are precisely our
penalty pairs in which neither caret is a left caret.  Guba computes
word length of an element to be the number of cells in the diagram
plus twice the number of special vertices, corresponding exactly to
our formula above.

It follows from Guba's length formula that we may always form a
minimal penalty tree consisting only of penalty and left carets when
$n=1$, providing an alternate proof of Lemma \ref{lemma:min-tree}.
The example given above shows that this penalty tree may not be
minimal when $n> 1$.

Now we compare our formula with the other two in the literature, due
to Belk and Brown \cite{BB} and Fordham \cite{F}, which are based on
tables of weights corresponding to the different caret types.
Encoded in each table is some of the information that we use when we
tabulate $l_{\infty}(g)$ for $g \in F$.

Belk and Brown \cite{BB} use forest diagrams for elements of $F$
which, roughly, enumerate the right (resp. left) subtrees of the
left (resp. right) carets in each tree, with a pointer to the root.
They define four caret types, and their formula for the word length
of $g \in F$ is $l_0(g)+l_1(g)$, where, translating from forest
diagrams into binary trees, we see that $l_1(g)$ is simply the
number of interior carets in the tree pair diagram. Then $l_0(g)$ is
a sum of weights determined by the caret types with values $0,1$ or
$2$, which are presented in a table. The weights in the first row
and column of their table count the number of left carets in the
tree pair diagram distinct from the root caret, a count which we
include as part of $l_{\infty}(g)$. The remainder of the table has a
weight of two corresponding to each of our penalty carets in which
neither caret is a left caret.  Thus the two formulae are
equivalent.

Blake Fordham \cite{F} defines seven types of carets in a tree and
forms the pairs of caret types analogous to Belk and Brown. He
presents a six by six table of weights corresponding to the pairs of
caret types.  Altering Fordham's table in the following way:
\begin{enumerate}
\item subtract one from the weight of
each pair of caret types containing a single caret which is not a
right caret, and
\item subtract two from the weight of
each pair or caret types containing no right carets,
\end{enumerate}
one obtains a table that has a weight of two for each pair of caret
types which we call a penalty pair, excluding those in which one
caret type is left.  Thus his entire table counts $l_{\infty}(g)$
and the penalty contribution $p_n(g)$ simultaneously.

\subsection{Proof of Theorem \ref{thm:length}}

We rely on the following lemma of Fordham to prove Theorem
\ref{thm:length}.  This lemma gives conditions under which a
function defined from a group $G$ to the nonnegative integers
computes the word length of elements of the group.

\begin{lemma}[\cite{F}, Lemma 3.1.1]\label{lemma:blake}
Given a group $G$, a generating set $X$, and a function $\phi:G
\rightarrow \{0,1,2, \cdots \}$, if $\phi$ has the properties
\begin{enumerate}
\item $\phi(Id_G) = 0$;
\item if $\phi(g) = 0$ then $g = Id_G$;
\item if $g \in G$ and $\alpha$ or $\alpha^{-1}$ is any element of $X$, then $\phi(g)-1
\leq \phi(g\alpha)$; and
\item for any non-identity element $g \in G$, there is at least one
$\alpha\in G$ with either $\alpha$ or $\alpha^{-1}$ in $X$ such that
$\phi(g \alpha) = \phi(g) - 1$,
\end{enumerate}
then $\phi(g) = l(g)$ for all $g \in G$, where $l(g)$ denotes the
word length of $g$ with respect to the generating set $X$.
\end{lemma}

We now prove Theorem \ref{thm:length} by showing that the function
$\phi_n(g) = l_{\infty}(g) + 2 p_n(g)$ for $g \in F$ satisfies the
conditions of this lemma.

\begin{proof}
Define the function $\phi_n(g) = l_{\infty}(g) + 2 p_n(g)$ for $g
\in F$.  We must show that this function satisfies all four
conditions of Lemma \ref{lemma:blake}.  Since the identity is
represented by a tree pair diagram consisting of a single caret in
each tree, it is easy to see that both $l_{\infty}(Id)$ and
$p_n(Id)$ equal zero, and thus the first condition is easily
satisfied.

If $\phi_n(g) = 0$, in particular $\l_{\infty}(g) = 0$, so $g$ the
tree pair diagram for $g$ has no carets which are not right carets.
Thus $g$ is the identity in $F$.

We now state two lemmas which are slight variations on the last two
conditions, and defer their proofs to the next two sections, as they
are somewhat tedious.

\begin{lemma}\label{L:change1}
For every $g\in F$ and $\alpha \in \it{X}_n$, $\phi_n(g\alpha) =
\phi_n(g) \pm 1$.
\end{lemma}

\begin{lemma}\label{L:reduce}
For every $g\in F$, there exists $\alpha \in X_n$ such that
$\phi_n(g\alpha) = \phi_n(g) -1$.
\end{lemma}

Together with the fact that $\phi_n(g) = 0$ if and only if $g = id$,
Lemma~\ref{L:change1} implies that $\phi_n(g) \leq l_n(g)$ and
Lemma~\ref{L:reduce} implies that $l_n(g) \leq \phi_n(g)$, and hence
Theorem~\ref{thm:length} follows.
\end{proof}

The proofs of Lemmas \ref{L:change1} and \ref{L:reduce} depend
heavily on the combinatorial rearrangement of subtrees of a tree
pair diagram caused by multiplication by a particular generator.
This is illustrated in Figure \ref{fig:mult}.  This figure shows how
the subtrees of the original diagram are rearranged under
multiplication by $x_0$ and $x_2$.  It may be necessary to add
carets to the tree pair diagram to perform this multiplication.  In
general, multiplication by $x_n$ performs the analogous
rearrangement at level $n$ along the right side of the first tree in
the diagram.

Before proving Lemmas \ref{L:change1} and \ref{L:reduce}, we show
that the change in $l_{\infty}$ is easily computed when $g \in F$ is
multiplied by a generator $\alpha=x_i^{\pm 1}$.

We first fix some notation.  Let $(T_-,T_+)$ be the reduced tree
pair diagram for $g \in F$, and $(S_-, S_+)$  the reduced tree pair
diagram for a generator $\alpha=x_i^{\pm 1}$.  The tree pair diagram
for $g\alpha$ is formed by taking (possibly) unreduced
representatives $(T_-',T_+')$ of $g$ and $(S_-',S_+')$ of $\alpha$
in which $S_+'=T_-'$.  The (possibly unreduced) tree pair diagram
for $g \alpha$ is then given by $(S_-',T_+')$.  Careful examination
reveals that this process results in three mutually exclusive
situations, and in each case we can keep track of the difference
between $l_{\infty}(g)$ and $l_{\infty}(g \alpha)$.
\begin{obs}\label{O:multbygen}
The multiplication described above results in exactly one of the
following situations:
\begin{enumerate}
\item $S_+$ is not a subtree of $T_-$, so $T_+' \neq T_+$.
This implies that $(S_-',T_+')$ must be a reduced tree pair diagram
for $g\alpha$, and that $l_{\infty}(g\alpha)=l_{\infty}(g)+1$.
\item $S_+$ is a subtree of $T_-$, so $T_+'=T_+$, and $(S_-',T_+)$ is a
reduced tree pair diagram for $g\alpha$. In this case, the change in
$l_{\infty}$ depends on $\alpha$:
\begin{enumerate}
\item If $\alpha=x_i^{-1}$, then $l_{\infty}(g\alpha)=l_{\infty}(g)+1$.
\item If $\alpha=x_i$, then $l_{\infty}(g\alpha)=l_{\infty}(g)-1$.
\end{enumerate}
\item $S_+$ is a subtree of $T_-$, so $T_+'=T_+$, and $(S_-',T_+)$ is not
a reduced tree pair diagram for $g\alpha$, then
$l_{\infty}(g\alpha)=l_{\infty}(g)-1$.
\end{enumerate}
\end{obs}

Since $l_{\infty}(g)$ is an important part of $\phi_n(g)$, the above
observation will play a major role in the proof of Theorem
\ref{thm:length}.

\section{Proof of Lemma~\ref{L:change1}}
\label{sec:lemma1} We now prove Lemma \ref{L:change1}, which states
that multiplication by any generator in $X_n$ or its inverse changes
the value of $\phi_n(g)$ by either $1$ or $-1$.  Recall that $g =
(T_-,T_+)$.

\begin{proof}
 First note that since $l_\infty(g\alpha)$ and $l_\infty(g)$ always differ by $1$,
we may assume without loss of generality that $l_\infty(g\alpha) =
l_\infty(g) - 1$.  To see why, assume that Lemma~\ref{L:change1}
holds whenever we have $l_\infty(g\alpha) = l_\infty(g)-1$, and
consider a pair $g$ and $\alpha$ with $l_\infty(g\alpha) =
\l_\infty(g)+1$. Set $h = g\alpha$ and $\beta = \alpha^{-1}$.  Then
$l_\infty(h\beta) = l_\infty(h)-1$, so Lemma~\ref{L:change1} holds
for $h \in F$ and the generator $\beta$. Therefore, $\phi_n(g) =
\phi_n(h\beta) = \phi_n(h) \pm 1 = \phi_n(g\alpha) \pm 1$, and thus
$\phi_n(g\alpha) = \phi_n(g) \pm 1$.

Let $g\in F$ and $\alpha \in X_n$.  Without loss of generality, we
now assume that $l_\infty(g\alpha) = l_\infty(g)-1$. It will suffice
to prove that $p_n(g\alpha) = p_n(g)$ or $p_n(g)+1$. We split the
proof into two cases depending on the exponent of $\alpha$.

\underline{{\bf Case 1:}} $\alpha = x_i^{-1}$.  In the tree pair
diagram $(S_-,S_+)$ for $\alpha$, the tree $S_+$ consists entirely
of a string of $i+2$ right carets. Notice that we must be in Case 3
of Observation~\ref{O:multbygen}, in which $S_+$ is a subtree of
$T_-$. Thus $T_-$ also has at least $i+2$ right carets.  In $T_-$,
let $v_1 < v_2 < \cdots < v_i < v_{i+1}< v_{i+2}$ be the infix
numbers of the first $i+2$ right carets, beginning with the root
caret.  As a separate subtree, this set of right carets has $i+3$
leaves, each of which may have a subtree of $T_-$ attached to it.
Let $A_j$ be the (possibly empty) subtree attached to the left leaf
of caret $v_j$, for $1 \leq j \leq i+2$.  Let $A_{i+3}$ be the
(possibly empty) subtree attached to the right leaf of caret
$v_{i+2}$.  Note that multiplication by $x_i^{-1}$ affects caret
$v_{i+1}$, rotating it from the right side of the tree to the
interior (or left in the case $i=0$. See Figure \ref{fig:lemma8} for
a diagram of $(S_-,S_+)$ and $(T_-,T_+)$.

\begin{figure}
\begin{center}
\includegraphics[width=5in]{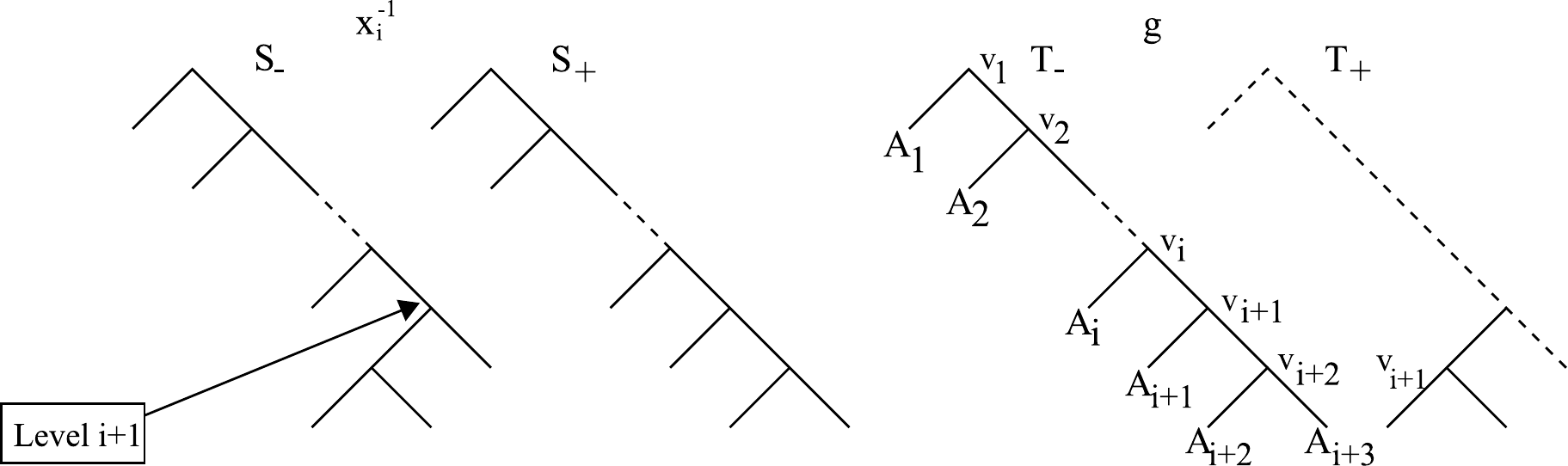}
\caption{Multiplication of $g = (T_-,T_+)$ by $\alpha = x_i^{-1} =
(S_-,S_+)$. We use dashed carets in $T_+$ to indicate that we do not
know {\em a priori} the exact shape of this tree, except for the
fact that $v_{i+1}$ is an interior caret of the given form.}
\label{fig:lemma8}
\end{center}
\end{figure}

Since we are in Case 3 of Observation \ref{O:multbygen},
multiplication of $(T_-,T_+)$ by $x_i^{-1} = (S_-,S_+)$ must create
an interior caret which is removed when the pair $(S'_-,T_+)$ is
reduced.  Thus we must have $A_{i+1}=A_{i+2}=\emptyset$, and that
caret $v_{i+1}$ is an exposed interior caret in $T_+$.  In addition,
if $A_{i+3}$ is also empty in $T_-$, then $v_{i+2}$ will also be
removed when the product $(S_-',T_+)$ is reduced.  Furthermore, if
for some $1 \leq k \leq i$, the subtrees $A_k, A_{k+1}, \ldots, A_i$
of $T_-$ are all empty, then carets $v_k, v_{k+1}, \ldots, v_i$ will
also all be removed when the product $(S_-',T_+)$ is reduced.

This removal of carets may cause certain other carets to alter their
penalty status, that is, penalty carets for $g$ may not be penalty
carets for $g\alpha$.  If $v_{i+1}$ is the only caret which is
removed by the reduction, then caret $v_i$ may change from being a
penalty caret for $g$ to not being a penalty caret
 for $g\alpha$. If more carets are removed during the reduction,
say $v_k, v_{k+1}, \ldots, v_{i+2}$ for $1 \leq k \leq i+1$, then
caret $v_{k-1}$ will switch from being a penalty caret in $g$ to a
non-penalty caret in $g\alpha$.

Suppose $\PP$ is any penalty tree for $g$. We claim that we can
always create a new tree $\PP'$ which is a penalty tree for
$g\alpha$ with $p_n(\PP') \leq p_n(\PP)$, which would imply that
$p_n(g\alpha) \leq p_n(g)$. There are two reasons that might prevent
$\PP$ itself from being a penalty tree for $g\alpha$:
\begin{itemize}
\item $\PP$ may contain vertices corresponding to carets in the
reduced tree pair diagram $(T_-,T_+)$ for $g$ which no longer appear
in the reduced tree pair diagram for $g \alpha$, or
\item there may be a leaf in $\PP$ corresponding to a penalty caret in
$(T_-,T_+)$ which is no longer a penalty caret in reduced tree pair
diagram for $g\alpha$.
\end{itemize}

Let us first consider the case that only the caret $v_{i+1}$ is
removed when $(S'_-,T_+)$ is reduced, and we describe how to alter
$\PP$ to create $\PP'$.

\begin{enumerate}
\item If $v_{i+1}$ does not appear as a vertex in $\PP$, and either $v_i$ does
not change penalty status as we go from $g$ to $g \alpha$, or $v_i$
does change penalty status, but is not a leaf in $\PP$, let
$\PP'=\PP$.

\item Suppose that $v_{i+1}$ does not appear as a vertex of $\PP$, $v_i$ does change penalty status,  and
$v_i$ is a leaf on $\PP$. In this case we form $\PP'$ by simply
removing the leaf $v_i$ and the edge connecting it to the tree, as
well as any newly exposed leaves which do not correspond to penalty
carets.

\item Suppose that $v_{i+1}$ does appear as a vertex of $\PP$.  We know that in
$g = (T_-,T_+)$, caret $v_{i+1}$ is not a penalty caret, since it is
a right caret in $T_-$ and an interior caret with no right subtree
in $T_+$.  Thus it cannot be a leaf of $\PP$, which forces $\PP$ to
have vertices $p$ and $q$ with $p \prec v_{i+1} \prec q$ for some
carets $p$ and $q$. But since $v_{i+1}$ is a right caret in $T_-$
with both $A_{i+1}$ and $A_{i+2}$ empty, and $v_{i+1}$ is an exposed
caret $T_+$, its generalized left and right edges are just the
actual left and right edges, so there is only one such caret $p$ and
one caret $q$, and hence $p=v_i$ and $q=v_{i+2}$. Construct $\PP'$
by removing the vertex $v_{i+1}$ from $\PP$, and adding the edge
$v_i \prec v_{i+2}$, since this adjacency exists in the reduced tree
pair diagram corresponding to $(S'_-,T_+)$ after caret $v_{i+1}$ is
removed. In this way, $v_i$ is not a leaf of $\PP'$ so its penalty
status, or any change therein, is irrelevant.
\end{enumerate}
In each case above, it is clear that $p_n(\PP') \leq p_n(\PP)$.

If more than one caret is removed when the tree pair diagram
$(S'_-,T_+)$ is reduced, say the string of carets $v_k, v_{k+1},
\ldots, v_{i+2}$ for some $1 \leq k \leq i+1$, the situation is
actually simpler. In this case carets $v_{k-1}, v_k, \ldots, v_i$
are all penalty carets for $g$, because they are right carets in
both trees, and are not the final caret in the diagram. Thus they
must appear as vertices of $\PP$, and using
Observation~\ref{O:bottleneck} one concludes that they must appear
in $\PP$ as a path $v_{k-1} \prec v_k \prec \cdots \prec v_i$, with
the vertex $v_i$ as the leaf.

In this case, we take $\PP'$ to be the tree $\PP$ with the string of
vertices from $v_{k-1}$ through $v_i$  removed.  It is possible that
some leaf of this tree which is created by the removal of these
vertices corresponds to a caret which is no longer a penalty caret
for $g \alpha$. In this case, this leaf may be removed, and the
resulting tree is a penalty tree for $g \alpha$.  It is clear again
that $p_n(\PP') \leq p_n(\PP)$, which implies that $p_n(g\alpha )
\leq p_n(g)$.

We now reverse the procedure outlined above to show that $p_n(g)
\leq p_n(g\alpha)$; namely we begin with a penalty tree $\PP'$ for
$g \alpha$ and describe how to alter it to obtain a penalty tree
$\PP$ for $g$ with $p_n(\PP) \leq p_n(\PP')$.  One of three things
may occur:
\begin{enumerate}
\item $\PP'$ may already be a penalty tree for $g$,
\item if caret $v_i$ changed penalty status between $g \alpha$ and
$g$, it may need to be added as a vertex of the penalty tree, if it
was not on $\PP'$, or
\item if caret $v_{k-1}$ changed penalty status between $g \alpha$ and
$g$, for some $k$ with $1 \leq k \leq i+1$, in which case the carets
 $v_k, \ldots, v_i$ were not present in the reduced tree pair
 diagram for $g \alpha$, the entire string of carets $v_{k-1} \prec v_k \prec \cdots \prec v_i$ may need to be
 added to form a penalty tree for $g$.
\end{enumerate}
Assume we are not in case (1), so we do need to add some of these
carets to $\PP'$. If we simply add the desired string of carets to
$\PP'$ to form $\PP$, we may increase the penalty quite a bit, but
the vertices of the tree which contribute to this increase must lie
on a path in $\PP'$ between $v_0$ and $p$, where $p$ is the caret at
the top of the newly added string. If any of these vertices do
become weighted, we alter the tree again in such a way that the only
vertices which are weighted in the new penalty tree but not weighted
in $\PP'$ must now lie between $v_0$ and some other vertex $q$,
where $q$ is closer to $v_0$ than $p$ was, and continue if necessary
until there are no more vertices which might switch from being
unweighted in $\PP'$ to being weighted in the altered tree $\PP$.

More precisely, to construct $\PP$, we will inductively construct a
series of trees $\PP'=\PP_0,\PP_1, \ldots \PP_{r}$ associated with
carets $v_{j_1}, \ldots, v_{j_r}$, a certain subset of the carets
$\{v_1, \ldots, v_{k-2}$, where $j_r < j_{r-1} \cdots < j_1$. For
each $r \geq 1$,  $\PP_r$ is a penalty tree for $g$, $\PP_r$
contains vertices corresponding to all carets $v_k$ where $j_r \leq
k \leq i$, and $j_r$ is the largest index $k$ with $k < j_{r-1}$ and
$v_k$  on $\PP_{r-1}$. In addition, either:
\begin{enumerate}
\item $p_n(\PP_{r}) \leq p_n(\PP')$, or

\item $p_n(\PP_{r}) > p_n(\PP')$ , $d_{\PP_{r}}(v_0,
v_{j_r})=d_{\PP'}(v_0, v_{j_r}) > j_r$, the reason that $p_n(\PP_r)$
exceeds $p_n(\PP')$ is that there are vertices along the path from
$v_0$ to $v_{j_r}$ in $\PP'$ which count towards $p_n(\PP_{r})$ but
not towards $p_n(\PP')$, and $v_i$ is always the leaf at maximal
distance from vertex $v_{j_r}$ in $\PP_r$.

\end{enumerate}

In the first case we take $\PP=\PP_{r}$, and in the second case, we
must construct $\PP_{r+1}$. But since $0 < j_r \leq i$ and $j_{r+1}
< j_r$, eventually case 1 above will occur, since $v_0 \in \PP'$ and
$d_{\PP'}(v_0,v_0)=0$. Hence, we can construct a penalty tree $\PP$
for $g$ with $p_n(\PP) \leq p_n(\PP')$, which implies that $p_n(g)
\leq p_n(g\alpha)$.

To complete the argument, we must show that the construction of
penalty trees $\PP_r$ satisfying the properties above is possible.
We first describe the construction of $\PP_1$. Let $j_1$ be the
largest index $j$ for which $v_j$ appears on the penalty tree
$\PP'$. Then we can attach a string of $i-j_1$ vertices and edges
corresponding to $v_{j_1} \prec \ldots \prec v_i$ to $\PP'$ to form
$\PP_{1}$, which is then a penalty tree for $g$. The added vertices
themselves will never be weighted, since $i<n$, but it is possible
that their addition might cause other unweighted vertices to become
weighted. Either this does not occur, so $p_n(\PP_{1}) \leq
p_n(\PP')$, or it does occur, so $p_n(\PP_{1}) > p_n(\PP')$, but
this only happens if the distance in $\PP'$ between vertices $v_0$
and $v_{j_1}$ satisfies $d_{\PP'}(v_0,v_{j_1})
> j_1$.  Moreover, it is only vertices along the path from $v_0$ to $v_{j_1}$
which may be weighted in $\PP_{1}$ but not in $\PP'$. Furthermore,
if there was a leaf of $\PP'$ further from $v_{j_1}$ than $v_i$ is,
then appending the new path to $v_i$ would not increase the total
penalty.

For the inductive step, suppose the penalty tree $\PP_{r-1}$ has
been constructed, and  $p_n(\PP_{r-1}) > p_n(\PP')$. Furthermore,
$v_i$ is the leaf at maximal distance from $v_{j_{r-1}}$ in
$\PP_{r-1}$, and the reason that $p_n(\PP_{r-1})$ exceeds $
p_n(\PP')$ is that there are vertices along the path from $v_0$ to
$v_{j_{r-1}}$ in $\PP_{r-1}$ which are weighted in $\PP_{r-1}$ but
not in $\PP'$. Then we construct $\PP_r$ as follows. Choose $j_r$ to
be the largest index $j$ with $0 \leq j < j_{r-1}$ so that $v_j$
corresponds to a vertex of $\PP_{r-1}$ (or equivalently, of $\PP'$).
Delete the first edge along the path connecting $v_{j_{r-1}}$ to
$v_0$ in $\PP_{r-1}$, and attach to $v_{j_r}$ the vertices and edges
corresponding to $v_{j_r} \prec \ldots \prec v_{(j_{r-1})-1}$, and
then add an edge connecting $v_{j_{r-1}-1}$ to $v_{j_{r-1}}$.  The
result, $\PP_{j_r}$, is clearly an allowable tree for $g$. Since
$v_i$ was the most distant leaf from $v_{j_{r-1}}$ in $\PP_{r-1}$,
$v_i$ is also the most distant leaf from $v_{j_{r-1}}, \ldots,
v_{j_r}$ in $\PP_r$. Now in $\PP_{r-1}$, only vertices between $v_0$
and $v_{j_{r-1}}$ may be weighted in $\PP_{r-1}$ but not in $\PP'$,
so deleting the edge connected to $v_{j_{r-1}}$ eliminates that
difference in penalty. None of the vertices between $v_{j_r}$ and
$v_{j_{r-1}}$ are close enough to a leaf to count towards $p_n$,
since they are too close to $v_i$, and $v_i$ is the most distant
leaf. Therefore, $p_n(\PP_{r}) \leq p_n(\PP')$ unless
$d_{\PP_{j_r}}(v_0, v_{j_r}) > j_r$, and then only vertices between
$v_0$ and $v_{j_r}$ can account for this increase in penalty. This
completes the desired construction, and thus the proof that $p_n(g)
\leq p_n(g\alpha)$.

Summing up, in this case where $\alpha=x_i^{-1}$, we have shown that
$p_n(g\alpha) \leq p_n(g)$ and $p_n(g) \leq p_n(g \alpha)$, and
hence $p_n(g)=p_n(g\alpha)$.

\underline{{\bf Case 2:}} $\alpha = x_i$. When $\alpha = x_i$ and we
are assuming that $l_{\infty}(g \alpha) = l_{\infty}(g) -1$, we must
be in either Case 3 or Case 2b of Observation~\ref{O:multbygen}.

To obtain the tree pair diagram for $\alpha = x_i$, we switch the
order of the trees given for $\alpha = x_i^{-1}$ in Figure
\ref{fig:lemma8}.  Thus $S_-$ is a tree consisting of a string of
$i+2$ right carets, and $S_+$ has a single caret which is not a
right caret: this caret is an interior caret if $i >0$ and a left
caret if $i=0$.

Since we are not in case 1 of Observation \ref{O:multbygen}, $S_+$
is a subtree of $T_-$. This guarantees an interior caret in $T_-$
which is the left child of the right caret at level $i$ from the
root.   As in Case 1, let $v_1 \prec v_2 \prec \cdots \prec v_i
\prec v_{i+2}$ be the first $i+1$ right carets in $T_-$, and let
$v_{i+1}$ be the interior caret hanging from the left leaf of caret
$v_{i+2}$. Number the leaves of the subtree consisting of the
$\{v_i\}$ from $1$ through $i+3$, and let $A_j$ be the (possibly
empty) subtree attached to leaf $j$.

If we are in Case 2(b) of Observation~\ref{O:multbygen} in which the
pair $(S'_-,T_+)$ is a reduced tree pair diagram, then there are two
carets which may change penalty status, as opposed to one in Case 3
of Observation~\ref{O:multbygen}.  In either case, the adjacency
$v_i \prec v_{i+2}$ which is present in $g = (T_-,T_+)$ may not
exist in the reduced tree pair diagram for $g \alpha$.

We claim first that $p_n(g\alpha) \leq p_n(g) + 1$, and begin our
argument by choosing a penalty tree $\PP$ for $g$. Below we
summarize the possible situations, which are not mutually exclusive,
which might force us to alter $\PP$ to obtain a penalty tree $\PP'$
for $g\alpha$.
\begin{enumerate}
\item $\PP$ contains the edge corresponding to $v_i \prec v_{i+2}$,
an adjacency present in $g$ but not in $g \alpha$, and the tree pair
diagram $(S'_-,T_+)$ is reduced. (Case 2(b) of
Observation~\ref{O:multbygen}.)
\item Caret $v_{i+1}$ is not a penalty caret for $g$, but is for
$g\alpha$, and the tree pair diagram $(S'_-,T_+)$ is reduced. (Case
2(b) of Observation~\ref{O:multbygen} and these conditions also
require that $A_{i+2} = \emptyset, A_{i+3} \neq \emptyset$, and
caret $v_{i+1}$ is a right caret which is not type $N$ in $T_+$.)
\item There is a single caret which is a penalty caret for $g$,
but no longer is one for $g\alpha$. This occurs as follows:
\begin{enumerate}
\item In either Case 2b of Observation~\ref{O:multbygen}, or Case 3
of Observation~\ref{O:multbygen} if exactly one caret is removed
when $(S'_-,T_+)$ is reduced, it may be the case that $v_i$ is a
penalty caret for $g$ but not for $g \alpha$.  This occurs if
$A_{i+1} = \emptyset$ and $v_i$ is a left or interior caret which is
not type $N$ in $T_+$.
\item In Case 3 of Observation~\ref{O:multbygen},
if carets $v_k, \cdots v_{i+1}, v_{i+2}$ are removed when
$(S'_-,T_+)$ is reduced, for some $0 \leq k \leq i+1$, then
$v_{k-1}$, if it exists, always changes from being a penalty caret
for $g$ to a non-penalty caret for $g \alpha$.
\end{enumerate}
\end{enumerate}

We describe a method for altering a penalty tree $\PP$ for $g$ into
a penalty tree for $g \alpha$ depending on which combination of the
above situations occurs.

Suppose first that the first situation does occur. Then either
$v_{i+1}$ is a vertex on $\PP$, or it is not. If $v_{i+1}$ is
already a vertex on $\PP$, then we delete both the edge
corresponding to $v_i \prec v_{i+2}$ as well as the edge along the
path from $v_0$ to $v_{i+1}$ which goes into $v_{i+1}$.  We
reconnect the tree by adding two edges corresponding to the
adjacency $v_i \prec v_{i+1} \prec v_{i+2}$. The resulting tree
$\PP'$ has vertices for all penalty carets for $g\alpha$.

We now claim that $p_n$ can increase by at most 1, and show this by
considering the distance from each vertex of the penalty tree to a
leaf of the penalty tree.  Recall that weighted  carets, that is,
those which count towards $p_n(\PP)$, are connected to the root of
the tree by a path of length at least two, and a leaf of the tree by
a path of length at least $n-1$.

In altering $\PP$ in this way to obtain $\PP'$, there are two carets
which might become weighted penalty carets. First,  it may be that
$v_{i+1}$ was not a weighted caret for $\PP$ but is weighted in
$\PP'$, since now all of the leaves which are connected by paths to
$v_{i+2}$ become leaves connected to $v_{i+1}$ also. This can happen
only if $d_{\PP}(v_{i+2},l) \geq n-2$ where $l$ is a leaf of $\PP$
at maximal distance from $v_{i+2}$. Second, it is possible that
there is a vertex $v$ along the path from $v_0$ to $v_i$ in $\PP$
which is not far enough from a leaf of $\PP$ to be weighted, yet
altering the tree by the addition of the edges $v_i \prec v_{i+1}
\prec v_{i+2}$ may now make this caret weighted. But this can happen
only if both $d_{\PP}(v_0, v) \geq 2$ and a leaf $l$ of $\PP$ which
has maximal distance from $v_{i+2}$ has $d_{\PP}(v,l)=n-2$. But this
implies that $d_{\PP}(v_{i+2},l) \leq n-3$.  Since these conditions
are mutually exclusive, we see that at most one of them can occur,
so $p_n(\PP') \leq p_n(\PP)+1$.

If, on the other hand, caret $v_{i+1}$ does not correspond to a
vertex of $\PP$, the situation is simpler. Simply delete the edge
$v_i \prec v_{i+2}$, and add a new vertex labeled $v_{i+1}$ along
with the edges $v_i \prec v_{i+1} \prec v_{i+2}$.  Again, remove
leaves as necessary until all remaining leaves correspond to penalty
carets of the tree pair diagram for $g \alpha$.  The resulting
penalty tree $\PP'$ for $g \alpha$ again satisfies $p_n(\PP') \leq
p_n(\PP)+1$.

Now if situation (2) also occurs, no additional alteration of the
penalty tree $\PP'$ is required, since $v_{i+1}$ is already on it.
Although situation (3a) may also occur, since $v_i$ is not a leaf of
$\PP'$, it does not concern us that it may no longer be a penalty
caret. However it is possible that some leaves of $\PP'$ may no
longer correspond to penalty carets in the reduced tree pair diagram
for $g \alpha$, since we may have created a new leaf when we removed
edges of $\PP$. Then we simply remove non-penalty leaves from $\PP'$
until all leaves do correspond to penalty carets.  This can never
increase $p_n(\PP')$.  Thus, $\PP'$ is a penalty tree for $g\alpha$
with $p_n(\PP') \leq p_n(\PP)+1$.

Now suppose that situation (1) above does not occur, but situation
(2) does. This implies that we are once again in Case 2(b) of
Observation~\ref{O:multbygen}, and hence the adjacency $v_i \prec
v_{i+2}$ is not present in $T_+$, which implies that $v_i \prec
v_{i+2}$ no longer holds for $g\alpha$. Therefore, since we assumed
that situation (1) does not occur, the edge $v_i \prec v_{i+2}$ does
not occur in $\PP$.  However, since $A_{i+3} \neq \emptyset$, it
follows that $v_{i+2}$ is a penalty caret for $g$, and hence must
appear on $\PP$.  However, the facts that $A_{i+2}= \emptyset$ and
$v_{i+1}$ is a right caret in $T_+$ which is not type $N$ imply that
the only two carets $v$ with $ v \prec v_{i+2}$ in $g$ are $v_i$ and
$v_{i+1}$. Since situation (1) does not occur,  the edge $v_{i+1}
\prec v_{i+2}$ is forced to exist in $\PP$, so caret $v_{i+1}$,
though not a penalty caret for $g$, was nonetheless already on
$\PP$.  Now if situation (3a) occurs, and $v_i$ is a leaf of $\PP$,
simply delete it. Continue to delete any non-penalty leaves from
$\PP$ to form a penalty tree $\PP'$ for $g\alpha$ with $p_n(\PP')
\leq p_n(\PP)$.

Finally, suppose that neither situations 1 nor 2 occur, but
situation 3 does. We must then be either in Case 2(b) or Case 3 of
Observation~\ref{O:multbygen}. First we consider what happens if we
are in case 2(b) of Observation~\ref{O:multbygen}. Then the only
reason $\PP$ might not be a penalty tree for $g \alpha$ is that
caret $v_i$ corresponds to a leaf of $\PP$, but $v_i$ is not a
penalty caret for $g \alpha$.  In this case, to form $\PP'$, we
delete the vertex corresponding to $v_i$ as well as any additional
leaves which no longer correspond to penalty carets in $g \alpha$.
The resulting tree satisfies $p_n(\PP') \leq p_n(\PP)$.

If we are in Case 3 of Observation~\ref{O:multbygen}, then some
carets are removed when the tree pair diagram $(S'_-,T_+)$ for $g
\alpha$ is reduced.  If these carets appear in $\PP$, we must delete
them when forming $\PP'$.  Once again,
Observation~\ref{O:bottleneck} reveals that these carets, if they
appear in $\PP$, appear as a string $v_k \prec \ldots \prec  v_i$ of
vertices, with $v_i$ as a leaf of the tree, and no other edges on
the tree out of any of these vertices.  Thus they can be easily
deleted, along with the vertex corresponding to caret $v_{k-2}$ if
necessary, to produce a penalty tree $\PP'$ for $g \alpha$ with
$p_n(\PP') \leq p_n(\PP)$.

Thus in all of these situations, we can always construct a penalty
tree $\PP'$ for $g \alpha$ with $p_n(\PP') \leq p_n(\PP)$, and it
follows that $p_n(g\alpha) \leq p_n(g)+1$.

We now prove that if we begin with a penalty tree $\PP'$ for $g
\alpha$, we can always alter it to construct a penalty tree $\PP$
for $g$ with $p_n(\PP) \leq p_n(\PP')$. If we are in Case 3 of
Observation~\ref{O:multbygen}, we must add vertices corresponding to
the carets $v_k,v_{k+1} \ldots, v_i$ to $\PP'$ to form $\PP$.  We do
this using the same inductive procedure used in Case 1 of the proof
of this lemma.

If we are in case 2(b) of Observation~\ref{O:multbygen}, there are
two possible situations to consider.
\begin{enumerate}
\item Caret $v_{i+1}$ is a penalty caret for $g\alpha$, but not for $g$.
This happens if $A_{i+2}=\emptyset$, caret $v_{i+1}$ is a right
caret in $T_+$ which is not type $N$, and $A_{i+3} \neq \emptyset$.
\item Caret $v_i$ is a  not a penalty caret for $g\alpha$,
but is a penalty caret for $g$. This occurs if $A_{i+1} = \emptyset$
and $v_i$ is a left or interior caret in $T_+$ which is not type
$N$.
\end{enumerate}
If the second situation above does not occur, or it does but $v_i$
corresponds to a vertex already on $\PP'$, then constructing $\PP$
from $\PP'$ requires only deleting any leaves which no longer
correspond to penalty carets in $g$.  This process cannot increase
the penalty weight of the tree.  If the second situation does occur,
and $v_i$ does not correspond to a vertex of $\PP'$, we again use
the inductive procedure from the first case of the proof of this
lemma to construct the desired penalty tree $\PP$ for $g$ containing
a vertex corresponding to $v_i$. Hence, $p_n(g) \leq p_n(g\alpha)$,
which in turn implies that in this case, either $p_n(g
\alpha)=p_n(g)$ or $p_n(g \alpha)=p_n(g)+1$, as desired.
\end{proof}

\section{Proof of Lemma~\ref{L:reduce}}
\label{sec:lemma2}

Before embarking on the proof itself, we gather together a few cases
in which $\phi_n(g\alpha)=\phi_n(g)-1$. We will show that any $g\in
F$ falls into at least one of these situations for some choice of
$\alpha$. As usual, we let $(T_-,T_+)$ be the reduced tree pair
diagram for $g$, and let $v_1 \prec v_2 \prec v_3 \prec \ldots \prec
v_j$ be all of the right carets in $T_-$ , and we let $A_k$ be the
(possibly empty) subtree attached to the left leaf of $v_k$ for $1
\leq k \leq j$. All of these observations essentially follow from
the proof of Lemma~\ref{L:change1}, and we supply details following
the statements below.

\begin{obs}\label{O:exposed}
For $0\leq i \leq n$, if $l_{\infty}(gx_i^{-1})=l_{\infty}(g)-1$,
then $\phi_n(gx_i^{-1})=\phi_n(g)-1$. This occurs precisely when
$T_-$ contains at least $i+2$ right carets,
$A_{i+1}=A_{i+2}=\emptyset$ in $T_-$, and caret $v_{i+1}$ is exposed
in $T_+$.
\end{obs}

\begin{obs}\label{O:absence}
For $0\leq i \leq n$, if $T_-$ contains at least $i+2$ right carets
and $A_{i+1} \neq \emptyset$, and there is a minimal penalty tree
$\PP$ for $g$ not containing $v_i \prec v_{i+1}$, then $\phi_n(gx_i)
= \phi_n(g)-1$.
\end{obs}

\begin{obs}\label{O:v2}
If there is a minimal penalty tree $\PP$ for $g$ in which the caret
$v_2$ is a weighted caret, then $\phi_n(gx_0^{-1})=\phi_n(g)-1$.
\end{obs}

The first two observations follow directly from the proof of
Lemma~\ref{L:change1}. Observation~\ref{O:exposed} falls into case 1
of the proof of  Lemma~\ref{L:change1}, and notice that in this case
we actually proved that $p_n(g)=p_n(gx_i^{-1})$, which implies
$\phi_n(gx_i^{-1})=\phi_n(g)-1$.  Now for
Observation~\ref{O:absence}, since the generator $\alpha=x_i$, we
look to case 2 of the proof. But the fact that there is a penalty
tree $\PP$ for $g$ not containing the edge $v_i \prec v_{i+1}$
corresponds to Situation 1 of the proof not occurring (Note that the
caret labeling is not the same as in the proof). As long as
situation 1 does not occur, $p_n(g)=p_n(gx_i)$.

Observation~\ref{O:v2} can be established by a similar type of
argument.  Note that the situation in Observation ~\ref{O:v2} is
distinct from the case $i=0$ in Observation~\ref{O:exposed}, for if
$v_1$ were exposed in $T_+$ and $A_1=A_2=\emptyset$ in $T_-$, then
$v_2$ must be a left caret in $T_+$, and thus is not a weighted
penalty caret. Hence, in the situation of Observation ~\ref{O:v2},
$l_\infty(gx_0^{-1}) = l_\infty(g) +1$. Given a minimal penalty tree
$\PP$ for $g$ in which $v_2$ is weighted,  we can construct a caret
tree $\PP'$ for $gx_0^{-1}$ by replacing the edge into $v_2$ by the
edge $v_0 \prec v_2$. In $\PP'$, $v_2$ is not weighted, and so
$p_n(\PP') \leq p_n(\PP)-1$, and hence $p_n(gx_0^{-1})\leq
p_n(g)-1$. This implies that $\phi_n(gx_0^{-1}) = \phi_n(g)-1$.

\begin{proof}[Proof of Lemma~\ref{L:reduce}]
Let $g \in F$ be represented by the reduced tree pair diagram
$(T_-,T_+)$. As usual, we let $v_1 \prec v_2 \prec \cdots\prec v_j$
be the right carets in $T_-$, and let $A_k$ be the (possibly empty)
subtree attached to the left leaf of $v_k$ for $1 \leq k \leq j$. We
proceed by analyzing two cases based on the number of right carets
in $T_-$ and the infix numbers of the penalty carets.

{\bf Case 1:} Either $T_-$ has at most $n+1$ right carets, or $T_-$
has more than $n+1$ right carets, caret $v_{n+1}$ is not a penalty
caret and there are no penalty carets above $v_{n+1}$ in the infix
ordering.

First, if $T_-$ consists entirely of right carets, then $T_+$ must
have an exposed caret $v_k$ where $k \neq j$, or else the tree pair
diagram would not be reduced.  But we claim $1 \leq k \leq n+1$, for
if $v_k$ is exposed in $T_+$ for $k > n+1$, then $v_{k-1}$ would be
a penalty caret with $k-1 \geq n+1$, contradicting the conditions of
this case. But then $\phi_n(gx_{k-1}^{-1}) = \phi_n(g)-1$ by
Observation~\ref{O:exposed}.

If $T_-$ has some carets which are not right, let $i$ be the
greatest index such that $A_i \neq \emptyset$. So $i \leq n+1$,
since neither $v_{n+1}$ nor caret beyond it are penalty carets.  Now
$v_{i-1}$ is type $N$ in $T_-$, but $v_i, v_{i+1}, \ldots , v_j$ are
all right carets which are not type $N$. Hence, in $T_+$, one of
$v_{i}, \ldots, v_j$ must not be a right caret, else the tree pair
diagram is not reduced. If there is some penalty caret at or beyond
$v_i$, then it must either be of type $N$ or a right caret in $T_+$,
and hence one of $v_{i}, \ldots, v_j$ must have type $N$ in $T_+$.
Let $v_k$ be the highest (in infix order) type $N$ caret in $T_+$;
since there are no penalty carets at or beyond $v_{n+1}$, $i \leq k
\leq n $. Then this implies that caret $v_{k+1}$ is an exposed caret
in $T_+$, and $i+1 \leq k+1 \leq n+1$, which implies by
Observation~\ref{O:exposed} that $x_{k}^{-1}$ reduces $\phi_n$.  If,
on the other hand, there are no penalty carets at or beyond $v_i$,
then $v_{i-1} \prec v_i$ is not on any caret tree for $g$, so by
Observation~\ref{O:absence}, $x_{i-1}$ reduces $\phi_n$.

{\bf Case 2:}  $T_-$ has at least $n+2$ right carets and there are
penalty carets at or above $v_{n+1}$ in the infix ordering.

In this case, if for some $0\leq i \leq n$, $A_{i+1} \neq \emptyset$
and there is a minimal penalty tree $\PP$ for $g$ not containing the
edge $v_i\prec v_{i+1}$, then by Observation~\ref{O:absence},
$\phi(gx_i) = \phi(g) -1$.  Furthermore, if $v_2$ is weighted in
some minimal penalty tree $\PP$ for $g$, then by
Observation~\ref{O:v2}, $\phi_n(gx_0^{-1}) = \phi_n(g)-1$.

So, we may assume that for every minimal penalty tree $\PP$ for $g$,
$v_2$ is not a weighted caret and for each $0 \leq k \leq n$ such
that $A_{k+1} \neq \emptyset$,  $\PP$ contains the edge $v_k \prec
v_{k+1}$. We split into subcases; in each subcase we will show that
Observation~\ref{O:exposed} applies for some $i$.

{\bf Subcase 2.1:} $v_2$ is a left caret in $T_+$.

In this subcase, Observation~\ref{O:exposed} applies with $i=0$. To
see this, first note that $A_2=\emptyset$, for if not, then every
minimal penalty tree for $g$ must contain the edge $v_1 \prec v_2$.
The proof of Lemma \ref{lemma:leftcarets} implies that we can always
construct a minimal penalty tree for $g$ which contains the edge
$v_0 \prec v_2$ and does not contain the edge $v_1 \prec v_2$.
Therefore $A_2=\emptyset$, and hence $v_1$ and $v_2$ are consecutive
carets with $v_2$ a left caret in $T_+$. So in $T_+$, $v_1$ is not a
caret of type $N$ or a right caret, and recall that in $T_-$, caret
$v_1$ is not of type $N$, so $v_1$ is not a penalty caret.
Furthermore, $v_1 \prec v_2$ is the only edge out of $v_1$. But this
implies that $A_1=\emptyset$, for if not, then by assumption all
minimal trees $\PP$ realizing $p_n(g)$ contain the edge $v_0 \prec
v_1$. Given such a $\PP$, $v_1$ cannot be a leaf since it is not a
penalty caret, so $\PP$ also must contain $v_1 \prec v_2$. Then
alter $\PP$ by removing $v_1$ along with both edges $v_0 \prec v_1$
and $v_1 \prec v_2$, and adding the edge $v_0 \prec v_2$, obtaining
a penalty tree $\PP'$ not containing $v_0 \prec v_1$ with $p_n(\PP')
\leq p_n(\PP)$. So $A_1=A_2=\emptyset$, $v_1$ must be a left caret
in $T_+$, and Observation~\ref{O:exposed} applies with $i=0$ to show
that $\phi_n(gx_0^{-1}) = \phi_n(g)-1$.

{\bf Subcase 2.2:} $v_2$ is not left in $T_+$, and $v_2 \notin \PP$
for some minimal penalty tree $\PP$ for $g$.

In this subcase, Observation~\ref{O:exposed} applies with $i=1$. To
see this, first note that $v_2 \notin \PP$ which implies that $v_1
\prec v_2 \notin \PP$ and hence that $A_2 = \emptyset$. Also, $v_2
\notin \PP$ implies that $v_2$ is not a penalty caret, which implies
that $v_2$ cannot have type $N$ in $T_-$, and hence $A_3 =
\emptyset$.

Moreover, since $v_2$ is not a penalty caret, it follows that $v_2$
is an interior caret in $T_+$ which is not of type $N$.  This
implies that $v_1$ is either of type $N$ in $T_+$ or is an interior
caret which is not of type $N$.  We claim that $v_1$ must be of type
$N$.  Suppose $v_1$ is interior, but not of type $N$.  It follows
that $A_1 \neq \emptyset$, which implies by our assumption that
$\PP$ contains the edge $v_0 \prec v_1$.  We know that $v_1$ is not
a penalty caret because it is not type $N$ in either tree, and is an
interior caret in $T_+$.  Thus $v_1$ is not a leaf of $\PP$, so
there must be some edge out of $v_1$ in $\PP$. The only possible
edge out of $v_1$ is $v_1 \prec v_2$, which means $v_2 \in \PP$, a
contradiction. Therefore $v_1$ has type $N$ in $T_+$, which in turn
implies that $v_2$ is exposed in $T_+$, and so by
observation~\ref{O:exposed}, $\phi_n(gx_1^{-1}) = \phi_n(g)-1$.

{\bf Subcase 2.3:} $v_2$ is not a left caret in $T_+$, and for every
minimal penalty tree $\PP$ for $g$, $v_2\in \PP$ but $v_2$ is not
weighted.

Choose a minimal penalty tree $\PP$ for $g$. Since $v_2$ is neither
a left caret nor weighted, it follows that $d_{\PP}(v_2,l) < n-1$
for all leaves $l$ of $\PP$. Now note that if all edges $v_k \prec
v_{k+1}$ for $2 \leq k \leq n$ are on $\PP$, then $d_{\PP}(v_2,
v_{n+1}) = n-1$, so $v_2$ would be at least distance $n-1$ from some
leaf of $\PP$. So let $i$ be the smallest index such that $v_i \prec
v_{i+1}$ is not on $\PP$. Hence, $A_{i+1} = \emptyset$, for
otherwise $v_i \prec v_{i+1}$ would be on $\PP$ by the conditions of
Case (2).  Note that $A_{i+1} = \emptyset$ means that the carets
$v_i$ and $v_{i+1}$ are consecutive in infix order.

Since $v_{i-1} \prec v_i$ is on $\PP$, $v_i \in \PP$. We claim that
$v_i$ must have type $N$ in $T_+$, otherwise $v_i \prec v_{i+1}$ is
the only possible edge out of $v_i$, so $v_i$ is a leaf of $\PP$.
But then $v_i$ must be a penalty caret, so must be a right caret in
$T_+$.  Since there must be some penalty caret $v$ beyond $v_i$, and
$v_i$ is a right caret in both trees, by
Observation~\ref{O:bottleneck}, the path in $\PP$ connecting $v$ to
$v_0$ must pass through $v_i$, contradicting the fact that $v_i$ is
a leaf of $\PP$. So $v_i$ has type $N$ in $T_+$, which implies that
$v_i \prec v_{i+1}$ is the only possible edge into $v_{i+1}$, so
$v_{i+1} \notin \PP$, so $v_{i+1}$ is not a penalty caret, and thus
must be an interior caret in $T_+$ which is not of type $N$, hence
exposed in $T_+$. Also, since $v_{i+1}$ is not a penalty caret, it
cannot have type $N$ in $T_-$, and hence $A_{i+2} = \emptyset$. So,
by observation~\ref{O:exposed}, $\phi_n(gx_i^{-1}) = \phi_n(g)-1$.
\end{proof}

\section{$(F,X_n)$ is not almost convex}
\label{sec:notAC}

A finitely generated group $G$ is {\em almost convex(k)}, or $AC(k)$
with respect to a finite generating set $X$ if there is a constant
$L(k)$ satisfying the following property.  For every positive
integer $n$, any two elements $x$ and $y$ in the ball of radius $n$
with $d_X(x,y) \leq k$ can be connected by a path of length $L(k)$
which lies completely within this ball.  Cannon, who introduced this
property in \cite{C}, proved that if a group $G$ is $AC(2)$ with
respect to a generating set $X$ then it is also $AC(k)$ for all $k
\geq 2$ with respect to that generating set. Thus if a group is
$AC(2)$, it is called {\em almost convex} with respect to that
generating set.  If a group is almost convex with respect to {\em
any} generating set, then we simply call it almost convex, omitting
the mention of a generating set.

There are interesting examples of families of groups with and
without this property.  Groups which are almost convex with respect
to any generating set include hyperbolic groups \cite{C} and
fundamental groups of closed 3-manifolds whose geometry is not
modeled on {\em Sol} \cite{SS}.  Moreover, amalgamated products of
almost convex groups retain this property \cite{C}.  Groups which
are not almost convex include fundamental groups of closed
3-manifolds whose geometry is modeled on {\em Sol} \cite{CFGT} and
the solvable Baumslag-Solitar groups $BS(1,n)$ \cite{MS}.

Almost convexity is a property which depends on generating set; this
was proven by Thiel using the generalized Heisenberg groups
\cite{T}.
 Cleary and Taback prove in \cite{CT} that Thompson's group
$F$ is not almost convex with respect to the standard generating set
$X_1=\{x_0,x_1\}$, but this has no implications for the convexity of
the group with respect to other generating sets.  Below we prove
that $F$ is not almost convex with respect to any consecutive
generating set $X_n = \{x_0,x_1, \cdots ,x_n\}$.   The proof below
follows the outline of \cite{CT}.

\begin{thm}\label{thm:notAC}
Thompson's group $F$ is not almost convex with respect to the
generating set $X_n = \{x_0,x_1, \cdots ,x_n\}$ .
\end{thm}

We begin with an overview of the proof of the theorem.  Assume that
$(F,X_n)$ is almost convex, and construct particular group elements
$gx_n$ and $gx_n^{-1}$  so that $l_n(g) = l_n(gx_n) + 1 =
l_n(gx_n^{-1}) + 1 = k+1$. Almost convexity guarantees a short path
$\gamma$ from $gx_n$ to $gx_n^{-1}$ which lies completely within the
ball of radius $k$. Label the right caret at level $n+1$ in the
reduced tree pair diagram for $g$ by $r_{n+1}$. Let $\gamma_i$ for
$0 \leq i \leq k$ denote the prefix of $\gamma$ of length $i$.  In
the tree pair diagram for $gx_n \gamma_i$, caret $r_{n+1}$ will
change type and level as $i$ increases. The salient point is that in
$gx_n$ the caret $r_{n+1}$ is the right caret at level $n+2$, and in
$gx_n^{-1}$ it is an interior caret of level $n+2$, which is the
left child of the right caret at level $n+1$. Thus there is a point
along $\gamma$ where the caret with label $r_{n+1}$ is again the
right caret at level $n+1$.  Suppose this happens when the prefix
$\gamma_m$ is applied to $gx_n$.  To prove the theorem, we show that
$g x_n \gamma_m \notin B(k)$, contradicting the assumption of almost
convexity.

\begin{proof}
Suppose that $(F,X_n)$ is almost convex.  Then there is a constant
$L$ so that elements $x,y \in B(k)$ with $d_{X_n}(x,y) = 2$ can be
connected by a path of length at most $L$ which is contained in
$B(k)$.

We now construct a group element $g$ by giving a reduced tree pair
diagram $(T_-,T_+)$, so that the elements $gx_n^{\pm 1}$ yield a
counterexample to this assumption.

{\bf Constructing $T_-$.} Let $ r_1 \prec \cdots \prec r_{2n+1}
\prec r_{2n+2}$ be the right carets of $T_-$, where $r_1$ is the
root caret. These carets form a subtree with $2n+2$ leaves; let
$A_i$ be the subtree of $T_-$ whose root is attached to the left
leaf of caret $r_i$. For $i \leq n+1$, we take $A_i$ to be the
complete tree with $L+1$ levels. When $ n+2 \leq i \leq 2n+2$, $A_i$
will be empty.

{\bf Constructing $T_+$.} The root caret of $T_+$ will be the caret
immediately preceding $r_{n+1}$ in infix order. The right carets of
the right subtree of this caret will be $r_{n+1} \prec r_{n+2} \prec
\cdots \prec r_{2n-1} \prec r_{2n} \prec r_{2n+2}$, with the left
subtree of $r_j$ empty for $n+1 \leq j \leq 2n$, and the left
subtree of $r_{2n+2}$ consisting of the single caret $r_{2n+1}$. The
caret $r_{2n+1}$ is added as an interior caret to ensure that the
pair of trees is reduced. All carets before $r_{n+1}$ in infix order
will be left carets in this tree, except for the caret with infix
number two, which will be an interior caret, again simply to ensure
that the tree pair diagram is reduced.

Figure \ref{fig:notAC} gives an example of a group element which is
of this form.

\begin{figure}
\begin{center}
\includegraphics[width=4in]{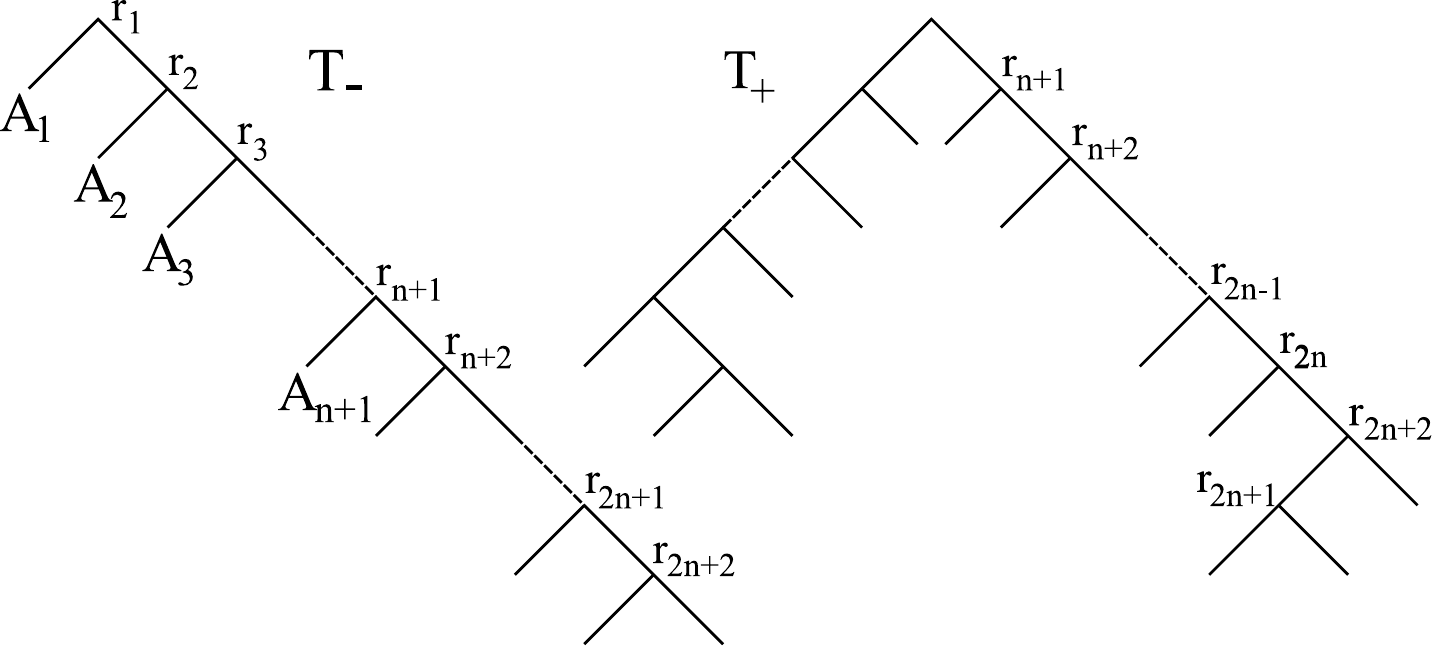}
\caption{An example of a group element $g$ constructed so that
$gx_n^{\pm 1}$ will contradict the assumption of almost convexity.}
\label{fig:notAC}
\end{center}
\end{figure}

We first prove a lemma which shows that both $x_n$ and $x_n^{-1}$
decrease the word length of $g$.  We then use $gx_n$ and $gx_n^{-1}$
as the two elements which will contradict the assumption of almost
convexity.

\begin{lemma}
Let $g = (T_-,T_+)$ be defined as above. Then
$l_n(gx_n)=l_n(gx_n^{-1})=l_n(g)-1$.
\end{lemma}

\begin{proof}
We show that multiplication by both $x_n$ and $x_n^{-1}$ decrease
the word length of the element $g$ constructed above.

{\bf Case 1. Multiplication by $x_n^{-1}$.} Multiplication by
$x_n^{-1}$ creates a pair of trees $(\tilde{T}_-,\tilde{T}_+)$ in
which the caret $r_{n+1}$ is now an interior caret in $\tilde{T}_-$,
and $\tilde{T}_+ = T_+$.  Thus $l_{\infty}(gx_n^{-1}) =
l_{\infty}(g) + 1$.

We will show that any penalty tree for $g$ can be altered to yield a
penalty tree for $gx_n^{-1}$ with one fewer weighted caret. First we
observe that all penalty carets in the reduced tree pair diagram for
$gx_n^{-1}$ were also penalty carets for $g$. The only two possible
differences in the reduced tree pair diagrams for $g$ and
$gx_n^{-1}$ which might influence the construction of penalty trees
are:
\begin{enumerate}
\item the caret $r_{n+1}$ is a right caret in both $T_-$ and $T_+$,
but in $\tilde{T}_-$ it becomes an interior caret which is not of
type $N$, hence is no longer a penalty caret in $gx_n^{-1}$, and
\item the adjacency $r_{n} \prec r_{n+2}$, not present for $g$, is present in $gx_n^{-1}$.
\end{enumerate}

Since $r_{2n}$ is a penalty caret for $g$, by
Observation~\ref{O:bottleneck}, the string of edges $ r_{n+1} \prec
\cdots \prec r_{2n}$ must appear in every penalty tree for $g$, and
$r_{n+1}$ is not a left caret in either $T_-$ or $T_+$. Hence
$r_{n+1}$ is a weighted caret in every penalty tree for $g$.
Furthermore, since $r_{n+1}$ and $r_{n+2}$ are consecutive carets in
the infix order, no other carets other than the $r_i$ carets in the
string above are connected to the root of the penalty tree by a path
passing through $r_{n+1}$.  Let $\mathcal{P}$ be any penalty tree
for $g$. Since caret $r_n$ is a left caret in $T_+$, we may assume,
by the proof of Lemma~\ref{lemma:leftcarets}, that if $r_n$ is a
vertex of $\PP$, then the edge $v_0 \prec r_n$ also appears on
$\PP$. We construct a penalty tree $\mathcal{P}'$ for $gx_n^{-1}$ as
follows: delete the edge $r_{n+1} \prec r_{n+2}$ from ${\mathcal
P}$. This leaves the caret $r_{n+1}$ as a leaf of ${\mathcal P}'$,
so we simply remove it, as it is not a penalty caret in $gx_n^{-1}$.
Now if $r_n$ did appear on $\PP$, connect $r_{n+2}$ via the edge
$r_n \prec r_{n+2}$. If not, add the two edges $v_0 \prec r_n \prec
r_{n+2}$. In either case, the number of weighted carets in the
subtree whose root is $r_{n+1}$ does not increase, and even if we
added the caret $r_n$ to $\PP$, it is not weighted. Thus caret
$r_{n+1}$, which was a weighted penalty caret in $\mathcal{P}$, is
not even present in $\mathcal{P}'$. Hence, $p_n(\mathcal{P}') \leq
p_n(\mathcal{P})-1$, so applying this argument to a minimal penalty
tree for $g$ yields $p_n(gx_n^{-1})=p_n(g)-1$, and we conclude that
$l_n(gx_n^{-1}) = l_n(g) - 1$.

 {\bf Case 2. Multiplication by
$x_n$.} Let $c$ be the caret which is the left child of caret
$r_{n+1}$ in $T_-$, that is, the root of the subtree $A_{n+1}$. Then
multiplication by $x_n$ produces a pair $(T_-',T_+')$ in which $c$
is now the right caret at level $n+1$ in $T_-'$, and $r_{n+1}$ is
the right caret at level $n+2$ in $T_-'$. Since an interior caret
has been changed to a right caret, $l_{\infty}(gx_n) =
l_{\infty}(g)-1$. Caret $r_{n+1}$, however, has not changed type: it
is of type $N$ in both $T_-$ and $\tilde{T}_-$, and a left caret
which is not of type $N$ in both $T_+$ and $\tilde{T}_+$. The only
other change is that the adjacency $r_{n} \prec r_{n+1}$ in $T_-$ no
longer exists in $T_-'$, and hence is not available for constructing
a minimal penalty tree. We will show that any penalty tree for $g$
may be altered to construct a penalty tree for $gx_n$ with no
additional weighted penalty carets.

Let $\mathcal{P}$ be any penalty tree for $g$. Let $c_{root}$ be the
root caret of $T_+$. As before, since $c_{root}$ is a left caret in
$T_+$, we may assume that either $c_{root}$ does not appear on
$\PP$, or if it does, so does the edge $v_0 \prec c_{root}$. If the
edge $r_{n} \prec r_{n+1}$ is not present in $\mathcal{P}$, then
$\mathcal{P}'=\mathcal{P}$ is a penalty tree for $gx_n$. If the edge
$r_{n-1} \prec r_n$ is present in $\mathcal{P}$, we construct
$\mathcal{P}'$ as follows. Delete the edge $r_{n} \prec r_{n+1}$ in
${\mathcal P}$.  If $c_{root}$ was on $\PP$, it appears on the edge
$v_0 \prec c_{root}$, and we add the edge $c_{root} \prec r_{n+1}$.
If $c_{root}$ was not on $\PP$, add it together with the two edges
$v_0 \prec c_{root} \prec r_{n+1}$ to form $\PP'$. Thus the vertices
$r_{n}$ and $r_{n+1}$ are present in both ${\mathcal P}$ and
${\mathcal P}'$.  It follows from the construction of ${\mathcal
P}'$ that $p_n(\mathcal{P}') \leq p_n(\mathcal{P})$, and hence
$p_n(gx_n)=p_n(g)$. Thus $l_n(gx_n) = l_n(g) - 1$ and the lemma
follows.
\end{proof}

It follows from the assumption that $(G,X_n)$ is almost convex that
there is a path $\gamma$ of length at most $L$ from $gx_n$ to
$gx_n^{-1}$ which is completely contained in the ball of radius $k$,
where $k = l_n(g)-1$.  We view $\gamma$ as a product
$\alpha_1\alpha_2 \cdots \alpha_L$ where each $\alpha \in \{x_0^{\pm
1},x_1^{\pm 1}, \cdots ,x_n^{\pm 1},Id\}$, and consider the prefixes
$gx_n\gamma_i=gx_n \alpha_1 \alpha_2 \cdots \alpha_i$.

We first consider the effect of multiplication by $x_n$ and
$x_n^{-1}$ on the caret $r_{n+1}$ in the initial word $g =
(T_-,T_+)$. This caret, in $T_-$, is a right caret at level $n+1$.
After multiplication by $x_n$, we obtain $gx_n = (T_-',T_+')$, and
now caret $r_{n+1}$ is a right caret in $T_-'$ at level $n+2$. After
multiplication by $x_n^{-1}$, we obtain $gx_n^{-1} =
(\tilde{T_-},\tilde{T_+})$, and this caret is an interior caret in
$\tilde{T_-}$ which is the left child of the right caret at level
$n+1$.

In each prefix $gx_n\gamma_i=gx_n \alpha_1\alpha_2 \cdots \alpha_i$
we note the position of the caret with label $r_{n+1}$. The
generators in the set $X_n$ and their inverses perform combinatorial
rearrangements of the subtrees of the tree pair diagram representing
$gx_n\gamma_i$ at levels one through $n+1$ along the right side of
the negative tree in the pair. Thus, there is a first point along
the path $\gamma$ at which caret  $r_{n+1}$ is again the right caret
at level $n+1$. Denote this prefix of $\gamma$ by $\beta$, which has
length $j$ where $1 \leq j \leq L$. Denote the prefixes of $\beta$
by $\beta_i$, where $1 \leq i \leq j$.

We note that because of the choice of $g$, multiplication of $gx_n
\beta_i$ by $\alpha_{i+1}$ never requires the addition of carets to
the tree pair diagram for $gx_n \beta_i$, and as a result, the
positive tree is always unchanged by this multiplication.
Additionally, after this multiplication is performed, no cancelation
is necessary to obtain the reduced tree pair diagram. The only
exposed carets in $T_+$ are in the second and the penultimate
carets, and these carets are not exposed in $T_-$, nor can they ever
become exposed along $\beta$. This means that the number of carets
in the tree pair diagrams for $gx_n \beta_i$ remains constant for $i
= 1,2 \cdots ,j$.

For each prefix $\beta_i$ of $\beta$, we consider the tree pair
diagram for $g_i = gx_n\beta_i$.  As the values of $i$ increase, the
position of caret $r_{n+1}$ moves up and down the right side of the
negative tree at levels at least $n+1$, and is unchanged in the
positive tree. If the next generator in the path $\beta$ is of the
form $x_j$, then the level of $r_n$ in the negative tree increases
by one. If the next generator in the path $\beta$ is $x_j^{-1}$,
then the level of $r_n$ in the negative tree decreases by one.  In
either case, the position of this caret in the positive tree is
unchanged. Since the level of caret $r_{n+1}$ must have a net
decrease of 1, the path $\beta$ necessarily consists of $m+1$
generators with negative exponents and $m$ generators with positive
exponents.

To prove this theorem, we show that generators of the form
$x_j^{-1}$ as part of the path $\beta$ always increase the word
length.  Thus the word length $l_n(gx_n \beta)$ satisfies the
following inequality:
$$l_n(gx_n\beta)
\geq l_n(gx_n) +(m+1)-m=k+1 > k.$$  It follows from this inequality
that the element $gx_n\beta$ does not lie in the ball of radius $k$,
contradicting the assumption of almost convexity.

Since multiplication by $x_j^{-1}$ will always move a right caret to
an interior or left caret, and carets are never added in order to
complete multiplication along the path $\beta$, multiplication of
$gx_n \beta_i$ by $x_j^{-1}$ will always yield $l_{\infty}(gx_n
\beta_{i+1}) = l_{\infty}(gx_n \beta_{i}) + 1.$

We now show that the penalty contribution to the word length is
unchanged when $gx_n \beta_i$ is multiplied by $x_j^{-1}$.  Each
such multiplication changes a right caret into an interior caret,
and also disrupts some adjacency, which might affect the penalty
tree.  However, we note two salient points:
\begin{enumerate}
\item The caret which is shifted from right to interior by this multiplication
always precedes caret $r_{n+1}$ in infix order, and any such caret
can be connected to the right side of the negative tree for $gx_n
\beta_i$ by a path of at most length $L$.  Thus such a carets is a
left caret in $T_+$ as well as in the positive tree in the reduced
pair representing $gx_n \beta_i$.
\item It follows from Lemma \ref{lemma:leftcarets} that this caret
is never a weighted penalty caret in any  minimal penalty tree for
$gx_n \beta_i$, for any $i$.
\end{enumerate}
Thus when $gx_n \beta_i$ is multiplied by $x_j^{-1}$ to obtain $gx_n
\beta_{i+1}$, we must have $p_n(gx_n \beta_i) = p_n(gx_n
\beta_{i+1})$. Combining this with the fact that $l_{\infty}(gx_n
\beta_{i+1}) = l_{\infty}(gx_n \beta_{i}) + 1$ implies that
$l_{n}(gx_n \beta_{i+1}) = l_{n}(gx_n \beta_{i}) + 1$, which proves
the theorem.
\end{proof}

\section{Depth of Pockets in $(F,X_n)$}
\label{sec:pockets}

Let $G$ be a finitely generated group with a finite generating set
$S$.  We say that $w \in G$, with $|w|_S = n$, is a $k$-pocket if
$B_w(k) \subset B_{Id}(n)$, taking the maximal $k$ for which this is
true.  Thus any path from $w$ in the Cayley graph $\Gamma(G,S)$ of
length at most $k$ remains in the ball of radius $n$ centered at the
identity, and there is some path of length $k+1$ emanating from $w$
which leaves this ball.  The integer $k$ is called the {\em depth}
of the pocket.

We say that a group $G$ has {\em deep pockets} with respect to a
finite generating set $S$ if there is no bound on the depth of group
elements.  Bogopol'sk$\breve{\textrm{i}}$
proved in \cite{B} that
hyperbolic groups have finite depth, that is, for every generating
set there is a uniform upper bound on the depth of all pockets.
There are many examples of finitely-generated infinite groups with
deep pockets:  the lamplighter groups $\Z_n \wr \Z = \langle a,t |
t^n, [a,a^{t^i}], \ i \in \Z \rangle$ with respect to the generating
set $\{a,t\}$ were the first examples of such groups \cite{CT3}, and
a finitely presented example of such a group is given in \cite{CR}.
Warshall, in \cite{W}, proves that the discrete Heisenberg group $
\langle x,y | [x,[x,y]],[y,[x,y]] \rangle$ has deep pockets with
respect to any finite generating set. Riley and Warshall in
\cite{RW} prove that the property of having deep pockets does depend
on the choice of generating set.

We show below that for any $k \in \Z^+$, Thompson's group $F$ has a
generating set $X_n = \{x_0,x_1, \cdots ,x_n\}$ which yields pockets
of depth at least $k$, as long as $n \geq 2k+2$. Since $2k+2$ is
always greater than one, this does not contradict the result in
\cite{CT} stating that $(F,X_1)$ has only pockets of depth two. The
theorem below is really of interest for large values of $k$.  It is
proved by example; for a given $k$ we construct a family of pockets
whose depth is at least $k$ with respect to $X_n$.  In \cite{CT}, an
exhaustive description is given of all pockets with respect to
$X_1$, which are necessarily of depth two.  We do not give such a
description below with respect to $X_n$.

In addition, we give upper bounds on pocket depth in each of these
generating sets. We show that for fixed a $n$, there are no pockets
of depth greater than or equal to the maximum of $4n-3$ and $2n+1$.
Note that for $n \geq 2$ we have $4n-3 \geq 2n+1$, so it is only for
the case $n=1$ that the upper bound on pocket depth is $2n+1=3$, and
in this $n=1$ case, there are in fact pockets of depth 2.

\begin{thm}\label{thm:pockets}
For any $k \geq 1$, Thompson's group $F$ has pockets of depth at
least $k$ with respect to the generating set $X_n = \{x_0,x_1,
\cdots ,x_n\}$, for $n \geq 2k+2$.
\end{thm}

\begin{proof}
We construct a group element $g=g_k = (T_-,T_+)$ for each $k \in
\Z^+$ which is a pocket of depth at least $k$ with respect to the
generating set $X_n$, for $n \geq 2k+2$ by describing the trees
$T_-$ and $T_+$. We assume that the carets of these trees are
numbered in infix order.

\begin{enumerate}
\item Let $ r_1
\prec \cdots r_{2n+k+2}$ be the right carets of $T_-$.  Let $A_i$ be
the left subtree of $r_i$; we choose  $A_i$ to be the complete tree
with $k+1$ levels for $1 \leq i \leq n+k+1$. For $i > n+k+1$, $A_i$
is empty.

\medskip

\item The right carets of $T_+$ are $r_1 \prec \cdots \prec r_{2n+k} \prec r_{2n+k+2}$, but caret $r_{2n+k+1}$ is the left child of caret $r_{2k+k+2}$,  an interior caret.
Denote the left subtree of caret $r_i$ by $B_i$, and as in $T_-$,
$B_i$ is empty for $n+k+1 < i \leq  2n+k$. For $1\leq i \leq n+k+1$,
as an independent tree, $B_i$ consists of a string of right carets,
one fewer in number than the number of carets in $A_i$, with a
penultimate interior caret.  This additional caret ensures that the
tree pair diagram will be reduced.

\end{enumerate}
Figure \ref{fig:pocket} gives an example of a group element of the
above form.

\begin{figure}
\begin{center}
\includegraphics[width=4in]{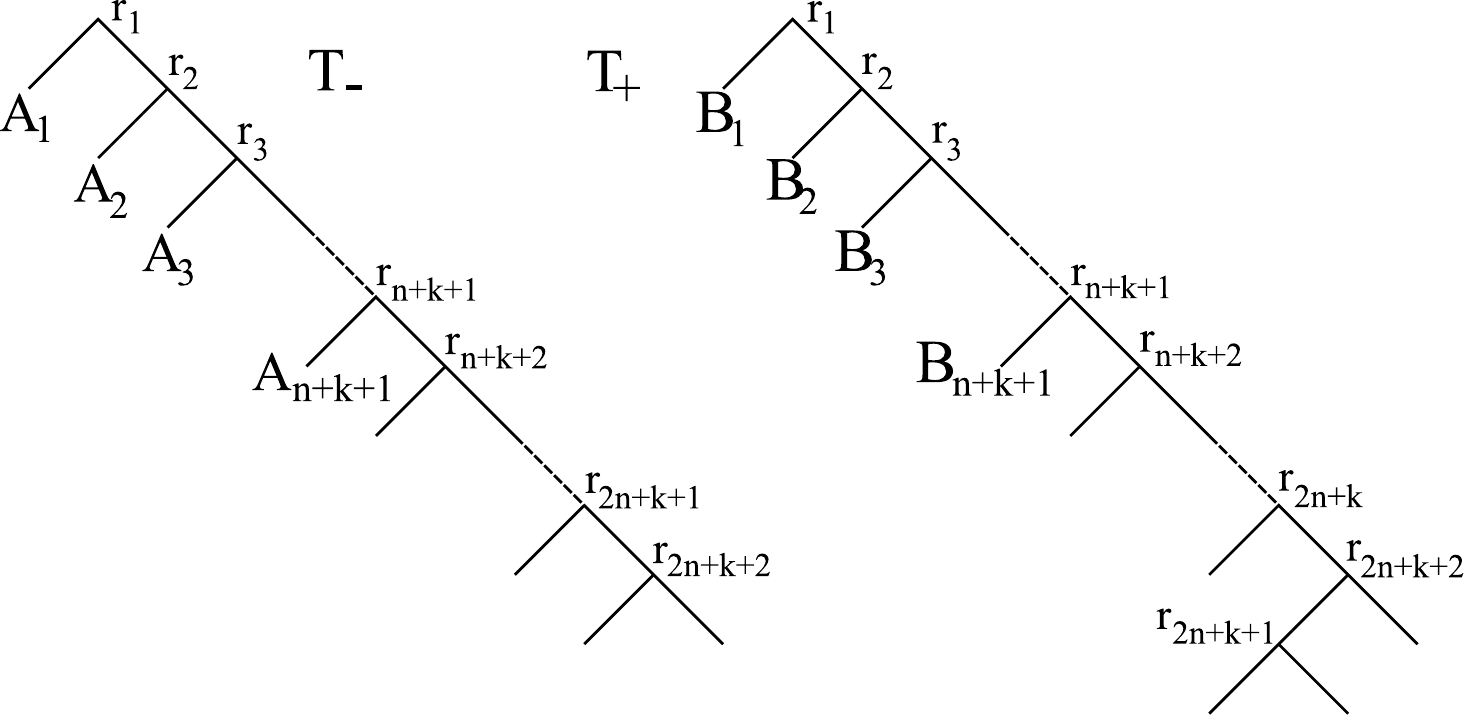}
\caption{An example of a group element which will be a pocket of
depth at least $k$.} \label{fig:pocket}
\end{center}
\end{figure}

Let $\beta=\alpha_1\alpha_2\ldots \alpha_k$ be any word with
$\alpha_i \in X_n$ or $\alpha_i^{-1} \in X_n$ for all $i$, and
denote the prefixes of $\beta$ by $\beta_i=\alpha_1 \alpha_2 \ldots
\alpha_i$.  The original word $g$ was constructed so that the
following are always true:
\begin{enumerate}
\item The original tree pair diagram $(T_-,T_+)$ is reduced.

\medskip

\item For each $i$, multiplication of $g \beta_i$ by $\alpha_{i+1}$ can
be accomplished without adding additional carets to the tree pair
diagram, and the resulting tree pair diagram for each $g\beta_i$ is
always reduced. Thus the number of carets in the reduced tree pair
diagram for $g \beta_i$ remains constant for $i=1,2, \cdots ,k$.

\medskip

\item In the tree pair diagram for $g \beta_i$, the positive tree
in the pair is always $T_+$, the same positive tree as in the
initial word $g$. Let $g \beta_i$ be represented by the reduced tree
pair diagram $(T_i,T_+)$.

\medskip

\item The only carets that can be affected when $g \beta_i$ is
multiplied by $\alpha_{i+1}$ are penalty carets.  Moreover, these
carets remain penalty carets when the multiplication is completed,
since they have type $N$ in $T_+$, and the tree $T_+$ is unchanged
by the multiplication.

\medskip
\item The subtree of $T_i$ with root caret $r_{n+k+2}$ remains unchanged for each
$g\beta_i$, and always hangs from the right leaf of caret
$r_{n+k+1}$. All carets in this subtree but the final two are
penalty carets, and necessarily form a string of length $n-1$ which
hangs from vertex $r_{n+k+1}$ in any penalty tree for $g\beta_i$, as
described in Observation \ref{O:bottleneck}.
\end{enumerate}

To prove Theorem \ref{thm:pockets}, we will show that $l_n(g\beta_i)
\leq l_n(g)$ for all $1 \leq i \leq k$.  We will describe the change
in $l_{\infty}$ between $g$ and $g \beta_i$, and bound the change in
penalty contribution between these two elements as well.

First note that a minimal penalty tree for $g$ is easily constructed
by joining each penalty caret to $v_0$ by choosing the shortest
adjacency path in the single tree $T_-$. Namely, connect each caret
to the caret adjacent to it via its honest, not generalized, left
edge. We call this path the {\em greedy path} from a caret to $v_0$.
It follows that the only penalty carets which are weighted in this
minimal penalty tree are $ r_2, \ldots r_{n+k+1}$, yielding $p_n(g)
= n+k$.

We begin with a lemma bounding the length of the greedy paths from
any caret to $v_0$.  This lemma is easily proved by induction.

\begin{lemma}\label{lemma:complete}
Let $T$ be any nonempty subtree of the complete tree on $m$ levels.
Then the maximum length of the greedy path from any caret to $v_0$
is $m$.
\end{lemma}

When considering possible penalty trees for $g \beta_i = (T_i,T_+)$,
we again must consider those carets on the right side the tree
$T_i$.  Let $M_i$ denote the number of carets $r_j$, for $1 \leq j
\leq 2n+k+2$,  which were right carets in $(T_-,T_+)$ but are no
longer right carets in $T_i$, and $N_i$ the number of right carets
in $T_i$ which are not amongst the carets numbered $r_j$ for $1 \leq
j \leq 2n+k+2$. Observe that
$l_{\infty}(g\beta_i)=l_{\infty}(g)+M_i-N_i$.

We give an upper bound for $p_n(g\beta_i)$ in order to control
$l_n(g\beta_i)$ by constructing a penalty tree for $g\beta_i$ which
is not necessarily minimal but will give the estimate necessary to
prove Theorem \ref{thm:pockets}.  We do this in two cases, depending
on the sign of $M_i-N_i$.

{\bf Case 1: $M_i-N_i > 0$.}  Construct a penalty tree ${\mathcal
P}_i$ for $g \beta_i$ once again by choosing the greedy paths in the
tree $T_i$. The right carets of $T_i$ are $c_1 \prec c_1 \prec c_2
\prec \cdots\prec c_l \prec r_{n+k+1} \prec r_{n+k+2} \prec \cdots
\prec r_{2n+k+2}$ where some subset of the first $l$ right carets
are equal to $r_j$ for values of $j$ between $1$ and $n+k$. These
adjacencies alone yield a subtree where each vertex, other than the
initial and final vertices, has valence two. For each $j$, the left
subtree of $c_j$ in the tree $T_i$ is a subtree of the complete tree
with $k+i+1$ levels, where $ i \leq k$. It follows from Lemma
\ref{lemma:complete} that the greedy path from a caret in the left
subtree of $c_j$ to $c_{j-1}$ has length at most $k+i+1$, where
$k+i+1 \leq 2k+1 \leq n-1$. Therefore we see that none of the carets
in the left subtrees of the $c_i$ correspond to weighted penalty
carets in ${\mathcal P}_i$. Thus $p_n(\mathcal{P}_i) =n+k-M_i+N_i =
p_n(g)-M_i+N_i$, and the difference in penalty contribution to the
word length between $g$ and $g \beta_i$ is bounded as follows:
$$p_n(g\beta_i) - p_n(g) \leq N_i - M_i .$$  Recall from above that
$l_{\infty}(g\beta_i)=l_{\infty}(g)+M_i-N_i$, and combine these
estimates to bound the difference in word length:

\begin{eqnarray}
\nonumber
  l_n(g\beta_i)-l_n(g) &=& \left(l_{\infty}(g\beta_i)-l_{\infty}(g)\right)
  +
2(p_n(g\beta_i)-p_n(g)) \\
 \nonumber  &=& (M_i-N_i) + 2(p_n(g\beta_i)-p_n(g)) \\
 \nonumber &\leq& (M_i - N_i) + 2(N_i - M_i) \\
 \nonumber  &=& N_i - M_i \\
  \nonumber  &<&0
\end{eqnarray}

It follows that when $M_i-N_i > 0$, we have $l_n(g\beta_i) <
l_n(g)$.

{\bf Case 2: $M_i-N_i \leq 0$ .} Unlike Case 1, we now build a
penalty tree ${\mathcal P}_i$ using first the adjacencies $r_j \prec
r_{j+1}$ present in $T_+$ for $1 \leq j\leq 2n+k$, attaching $r_1$
to the dummy caret $v_0$.  This again yields a tree where each
vertex other than the final and initial ones has valence two.

We now attach vertices to ${\mathcal P}_i$ representing the other
penalty carets of $T_i$, those not amongst the carets $r_j$ for $1
\leq j \leq 2n+k$. For each such caret $p$, we use the adjacencies
along the greedy path in $T_i$ from $p$ to $v_0$. We take the
longest subpath of the greedy path containing $p$ but none of the
$r_j$ carets, and attach vertices and edges corresponding to these
adjacencies to $P_i$, joining this path to the existing tree at the
next caret along the path, which is necessarily either $v_0$ or
$r_j$ for some $1 \leq j \leq 2n+k$.  We claim that the distance
between $p$ and that $r_j$ caret is at most $2k+1 \leq n-1$. This
will imply that none of these other carets $p$ will be weighted
carets in $\mathcal{P}_i$. To see why the claim is true, note that
if caret $p$ is a right caret in $T_i$, then the distance along the
greedy path to the next $r_j$ caret is at most $i \leq k$. If $p$ is
not a right caret, then it is in the left subtree of a right caret
$p'$ of $T_i$. The caret $p'$ is the right child of a caret $q$,
where $q$ is either a right caret of $T_i$ or the dummy caret $v_0$,
and the greedy path from $p$ to $v_0$ passes through $q$. If the
distance from $q$ to the next $r_j$ caret along that greedy path is
$m \leq i \leq k$, then the left subtree of $p'$ is a subtree of a
complete tree with $k+(i-m)+1$ levels. It follows from Lemma
\ref{lemma:complete} that the greedy path from $p$ to $q$ has length
at most $k+(i-m)+1$. Hence, the greedy path from $p$ to an $r_j$
caret has length at most $k+(i-m)+m+1=k+i+1 \leq 2k+1$, establishing
the claim. Therefore, $p_n(\mathcal{P}_i)=p_n(\mathcal{P} )$, and
hence $p_n(g \beta_i) \leq p_n(g)$.

We bound the difference in word length between $g$ and $g \beta_i$
as above, again using the fact that  $l_{\infty}(g \beta_i) =
l_{\infty}(g) +M_i - N_i$.

\begin{eqnarray}
\nonumber
  l_n(g\beta_i)-l_n(g) &=& \left(l_{\infty}(g\beta_i)-l_{\infty}(g)\right)
  +
2(p_n(g\beta_i)-p_n(g)) \\
 \nonumber  &=& (M_i-N_i) + 2(p_n(g\beta_i)-p_n(g)) \\
 \nonumber &\leq& (M_i - N_i) + 0 \\
 \nonumber  &=&  M_i - N_i\\
  \nonumber  &\leq&0
\end{eqnarray}

This shows that $g$ is a pocket of depth at least $k$ and completes
the proof of the theorem.
\end{proof}

Finally, we establish an upper bound on pocket depth.

\begin{thm}
For $n \geq 1$, $F$ has no pockets of depth $k$ with respect to
$X_n$, if $k \geq Max \{4n-3, 2n+1\}$.
\end{thm}

\begin{proof}
We will show that for every $g \in F$, at least one of
$l_n(gx_{i})$, for $0 \leq i \leq 2n$, or $l_n(g\alpha)$, where
$\alpha=x_{2n-1}^{-1}x_{2n-2}^{-1} \cdots x_2^{-1}x_1^{-1}$ is
greater than $l_n(g)$. Since $l_n(x_i) \leq l_n(x_{2n})=2n+1$ for $0
\leq i \leq 2n$, and $l_n(\alpha)=4n-3$, this proves the theorem.

Let $g \in F$ be represented by the reduced pair diagram $(T_-,
T_+)$, and let $r_1 \prec r_2 \prec \cdots \prec r_l$ be the right
carets of $T_-$, and let $A_i$ be the left subtree of $r_i$ for $1
\leq i \leq l$. First observe that, for $1 \leq i \leq 2n$, if $l
<i+1$ or if both $l \geq i+1$ and $A_{i+1} = \emptyset$, then
$l_n(gx_{i}) > l_n(g)$. Thus we need only consider the case that $l
\geq 2n+1$ and $A_{1}, A_{2}, \ldots A_{2n+1}$  are all not empty.

Assume that we are in this case; we will show below that it follows
that $l_n(g\alpha)> l_n(g)$. Note that in this case, the reduced
tree pair diagram for $g\alpha$ is $(T_{\alpha},T_+)$. In
$T_{\alpha}$, carets $r_i$ for $2 \leq i \leq 2n$ are all interior,
whereas they were right carets in $T_-$, so
$l_{\infty}(g\alpha)=l_{\infty}(g)+2n-1$. To compare penalty weight
between $g$ and $g\alpha$, notice that all of the $r_i$ carets which
are interior in $T_{\alpha}$ are type $N$, so they remain penalty
carets for $g\alpha$. The only change is in the caret adjacencies;
the adjacencies $r_i \prec r_{2n+1}$, for $2 \leq i \leq 2n-1$ are
present in $T_{\alpha}$, but not in $T_-$. We claim that $p_n(g)
\leq p_n(g\alpha)+n-1$.

To prove this claim, suppose $\PP$ is a minimal penalty tree for
$g\alpha$. We will construct a penalty tree $\PP'$ for $g$ as
follows.  If $\PP$ contains no edges of the form $r_i \prec
r_{2n+1}$, for $2 \leq i \leq 2n-1$, then $\PP'=\PP$ is a penalty
tree for $g$, so $p_n(g) \leq p_n(g\alpha)$. If $\PP$ does contain
one such edge, it contains only one, say $r_i \prec r_{2n+1}$. Then
alter $\PP$ to form $\PP'$ by deleting the edge, and inserting the
edge $r_{2n } \prec r_{2n+1}$, noting that $r_{2n}$ was already a
vertex on $\PP$, since it is has type $N$ in $T_{\alpha}$. Then
$\PP'$ is a penalty tree for $g$. It is possible that $p_n(\PP') >
p_n(\PP)$, but this possible increase can only be caused by carets
along the path from $v_0$ to $r_{2n}$ which were not weighted in
$\PP$ but become weighted in $\PP'$. However, there can be at most
$n-1$ of these, so $p_n(\PP') \leq p_n(\PP)+n-1$, and therefore
$p_n(g) \leq p_n(g\alpha)+n-1$.  Thus we obtain the necessary
inequality:

\begin{eqnarray}
\nonumber
  l_n(g\alpha) &=& (l_{\infty}(g\alpha)+ 2p_n(g\alpha)) \\
 \nonumber  &\geq& l_{\infty}(g)+ (2n-1) + 2(p_n(g)-n+1) \\
 \nonumber &=& l_n(g)+(2n-1)+2(1-n) \\
 \nonumber  &=& l_n(g) +1
\end{eqnarray}
which proves the theorem.
\end{proof}

\end{document}